%% file: MainHundleyShen.tex
\documentclass[reqno,12pt]{amsart}
\newtheorem{thm}{Theorem}[section]

\newtheorem{cor}[thm]{Corollary}
\newtheorem{lem}[thm]{Lemma}

\newtheorem{pro}[thm]{Proposition}
\newtheorem{rem}[thm]{Remark}

\usepackage[letterpaper]{geometry}
\usepackage{setspace}
\usepackage{lmodern}
\usepackage{mathdots}
\usepackage[T1]{fontenc}
\usepackage{amssymb}
\usepackage{mdwlist}
\usepackage{multirow}
\usepackage{amsmath,amssymb,epsfig}
\usepackage{amsmath}
\usepackage{type1cm}
\usepackage{url}
\usepackage{enumerate}
\usepackage[all]{xy}
\usepackage{epsfig}
\usepackage{color}

\include{shortcuts}

\begin{document}
\include{Hundley-Shen}

\end{document}

%% file: shortcuts.tex
\newcommand{\bpm}{\begin{pmatrix}}
\newcommand{\epm}{\end{pmatrix}}
\newcommand{\bsm}{\begin{smallmatrix}}
\newcommand{\esm}{\end{smallmatrix}}
\newcommand{\bspm}{\left(\begin{smallmatrix}}
\newcommand{\espm}{\end{smallmatrix}\right)}
\newcommand{\bm}{\begin{matrix}}
\renewcommand{\em}{\end{matrix}}
\newcommand{\bbm}{\begin{bmatrix}}
\newcommand{\ebm}{\end{bmatrix}}
\newcommand{\whitt}{\c W_{\psi_N}(\pi)}

\newcommand{\bs}{\backslash}

\newcommand{\lt}{{}\,_{t}}
\usepackage{MnSymbol}

\newcommand{\mnsw}[1]{{{}^2 \wedge_{#1}}}
\newcommand{\Sp}{\wt{Sp}}
\usepackage{hyperref}

\newcommand{\C}{\mathbb{C}}
\newcommand{\G}{\mathbb{G}}
\newcommand{\A}{\mathbb{A}}

\newcommand{\Z}{\mathbb{Z}}
\newcommand{\R}{\mathbb{R}}

\newcommand{\Tr}{\operatorname{Tr}}

\newcommand{\diag}{\operatorname{diag}}
\newcommand{\pr}{\operatorname{pr}}

\newcommand{\Ind}{\operatorname{Ind}}

\newcommand{\sym}{\operatorname{sym}}

\newcommand{\nInd}{\ul{\on{Ind}}}

\newcommand{\la}{\langle}
\newcommand{\ra}{\rangle}

\newcommand{\ve}{\varepsilon}
\newcommand{\vph}{\varphi}

\newcommand{\on}{\operatorname}
\newcommand{\ol}{\overline}
\newcommand{\ul}{\underline}
\renewcommand{\Re}{\on{Re}}

\newcommand{\gm}{\gamma}

\newcommand{\sg}{\sigma}

\newcommand{\il}{\int\limits_}
\newcommand{\iq}[1]{\il{[#1]}}

\newcommand{\wt}{\widetilde}
\newcommand{\Mat}{\on{Mat}}
\newcommand{\f}{\mathfrak}
\renewcommand{\c}{\mathcal}

\newcommand{\ud}[1]{\, d #1\,}
\newcommand{\II}{I\!\!I}
\newcommand{\III}{I\!\!I\!\!I}
\newcommand{\Flat}{\on{Flat}}
\newcommand{\Jac}{\on{Jac}}

%% file: Hundley-Shen.tex

\title[Twisted Spinor $L$-function]
{A Multi-variable Rankin-Selberg Integral for a 
Product of $GL_2$-twisted Spinor $L$-functions}
\author{Joseph Hundley, Xin Shen}
 \address{
 Department of Mathematics\\
244 Mathematics Building \\
University at Buffalo\\
Buffalo, NY 14260-2900}
\address{jahundle@buffalo.edu}
\address{
 Department of Mathematics\\
University of Toronto \\
Bahen Centre \\
40 St. George St., Room 6290\\
Toronto, Ontario \\
CANADA \\
M5S 2E4}
 \address{shenx125@math.utoronto.ca}
\thanks{This paper was written while the first named author was supported
by NSF Grant DMS-1001792.}

\maketitle
\section{Introduction}
\label{sec: intro}
An important problem in the theory of automorphic 
forms is to understand periods, and how they are 
related with $L$-functions and their special 
values, as well as with functorial liftings.
A prototypical example for this is 
the connection between
symmetric and exterior square
$L$-functions, functorial liftings from classical 
groups to $GL_n,$ and certain periods, 
which dates back at least to 
 \cite{JacquetShalika-Ext2}, 
 and is more fully explicated in 
\cite{GRS-PeriodsPoles}, and 
\cite{GRS-Book}.
Some more exotic examples are found, 
for example, in \cite{GH-spin8}, \cite{GH-spin10xsl2},
and a general framework which extends 
beyond classical groups is discussed in 
  \cite{GJS-PeriodsPoles}.
  
The connection between poles and liftings is well-
understood, at least philosophically: one 
expects that the $L$ function attached to a generic
cuspidal representation $\pi$ and a finite dimensional
representation $r$ of the relevant $L$-group will 
have a pole at $1$ if and only if the stabilizer of a 
point in general position for the representation $r$ is 
reductive and $\pi$ is in the image of the functorial lifting 
attached to the inclusion of the stabilizer of such a point.
For example, in the exterior square representation of $GL_{2n}(\C),$ the stabilizer of a point in general position
is $Sp_{2n}(\C),$ so one expects that a pole of the exterior 
square 
$L$ function indicates cuspidal representations which are
lifts from $SO_{2n+1}.$ 
The connection with periods is up to now less well 
understood.

In order to prove the expected relationship between 
poles and liftings in specific examples, and in 
order to draw periods into the picture, it is 
useful, perhaps essential, to have some sort of an 
analytic handle on both the $L$-function and the 
lifting. An analytic handle on the $L$-function may 
be provided by an integral representation, either 
of Langlands-Shahidi type or otherwise (integral 
representations which are not of Langlands-Shahidi 
type are often termed ``Rankin-Selberg''). 
An analytic handle on the lifting may be provided 
by an explicit construction.

Integral representation of $L$ functions and explicit construction of liftings between automorphic 
forms on different groups are important subjects 
in their own right as well. For example, integral 
representations are, as far as we know, the only way to 
establish analytic properties of $L$ functions in new 
cases. When an 
integral representation produces $L$ functions whose
analytic properties are already well understood, it 
nevertheless provides a new insight into the connection 
with periods, and identities among periods which can 
be otherwise quite surprising. This is the case 
in the present paper.

For explicit construction of liftings, there are two main 
ideas we know of. Each is related to the other and 
both are related to the theory of Fourier coefficients
attached to nilpotent orbits (\cite{GRS-FoCo},\cite{conj-unip}). 

The first main idea is to use a ``small'' representation 
as a kernel function. The prototypical example of this 
type is the classical theta correspondence  (\cite{Howe}).
In this type of construction, an automorphic form, 
which is defined on a large reductive group $H$ 
is restricted to a 
pair of commuting reductive subgroups, and 
integrated against automorphic forms on one member
of the pair to produce automorphic forms on the other.
In general, there is no reason such a construction 
should preserve irreducibility, much less be 
functorial. The right approach 
seems to be to take automorphic 
forms on $H$ which only support Fourier coefficients attached to very small 
nilpotent orbits.
For example,
a theta function, defined on the 
group $\wt{Sp}_{4mn}(\A)$ 
only supports Fourier coefficients attached to the minimal 
nilpotent orbit of this group. 
Its restriction to 
$Sp_{2n}(\A) \times O_{2m}(\A)$ provides a kernel 
for the theta lifting between these groups. 
Functoriality of this lifting was established in 
\cite{rallis}. This method has enjoyed brilliant 
success, but also has significant limitations. 
For example, it is not at all clear how the classical theta
correspondence could be extended to other groups 
of type $C_n\times D_m$: the embedding 
into $Sp_{4mn}$ is specific to $Sp_{2n} \times O_{2m}.$ 
The results of this paper hint at a possible way around this
difficulty.

The second main idea in 
explicit construction of correspondences
 is the descent method
of Ginzburg, Rallis, and Soudry (\cite{GRS-Book}, see
also \cite{Hundley-Sayag}). This construction treats the 
Fourier coefficients themselves essentially as 
global twisted Jacquet modules, mapping representations
of  larger reductive groups
to representations
of smaller reductive groups. 
As before, in general there is no reason 
for this construction to respect irreducibility, much less
be functorial, and a delicate calculus involving 
Fourier coefficients seems to govern when 
it is.

In this paper we define and study two new multi-variable 
Rankin-Selberg integrals, which are defined on the 
similitude orthogonal groups $GSO_{12}$ and $GSO_{18}.$  
These integrals 
are similar to  those
considered in  
\cite{BG-AdGL4},
\cite{G96-AdGSp4},
\cite{G95-Sym4GL2},
\cite{GH-G2Doubling},
\cite{GH-F4Survey},
in that each
involves applying a Fourier-Jacobi 
coefficient to a degenerate Eisenstein series and then 
pairing the result with a cusp form defined on 
a suitable reductive subgroup.  To be precise, 
$GSO_{6n}$ has a standard parabolic subgroup $Q$ whose Levi  is isomorphic 
to $GL_{2n} \times GSO_{2n}.$ 
The unipotent radical is a two step nilpotent group and the set
of characters of the 
center may be thought of as the exterior square representation 
of $GL_{2n}$ twisted by the similitude factor of $GSO_{2n}.$  
The stabilizer of a character in general position 
is isomorphic to 
$$C:=
\{ (g_1, g_2) \in GSp_{2n} \times GSO_{2n}: 
\lambda( g_1) = \lambda( g_2^{-1})\}.
$$
Here, $\lambda$ denotes the similitude factor.  
The choice of a character in general position as above 
also determines a projection of 
the unipotent radical onto a Heisenberg 
group in $4n^2+1$ variables, and a compatible 
embedding of $C$ into $Sp_{4n^2}.$

Our Fourier-Jacobi coefficient defines a map from automorphic 
functions on $GSO_{6n}(\A)$ to automorphic functions
on $C(\A).$  In the case $n=2$ and $3$ we apply this coefficient
to a degenerate Eisenstein series on $GSO_{6n}(\A)$ induced 
from a character of the parabolic subgroup $P$ 
whose Levi factor is isomorphic to $GL_3 \times GL_{3n-3} \times GL_1.$  We then pair the result with a pair of cusp forms 
defined on $GSp_{2n}(\A)$ and $GSO_{2n}(\A)$ respectively.  
The results suggest an intriguing connection with the theta correspondence for similitude groups.  

Indeed, in the case $n=2,$ the global integral turns out to 
be Eulerian, and to give an integral representation 
of 
$$
L^S( s_1, \wt \Pi \times \tau_1) L^S( s_2, \wt\Pi \times \tau_2),
$$
where $\Pi$ is a generic cuspidal automorphic representation 
of $GSp_4(\A)$ and $\tau_1, \tau_2$ are two (generic) cuspidal 
automorphic representations of $GL_2(\A)$ having the 
same central character, so that $\tau_1 \otimes \tau_2$ is 
a (generic) cuspidal 
automorphic representation of $GSO_4(\A).$  
It follows that the original integral has poles along both the
plane $s_1=1$ and the plane $s_2=1$ if and only if 
$\Pi$ is the weak lift of $\tau_1 \otimes \tau_2$ corresponding to 
 the 
embedding 
$$
GSpin_4(\C) = \{ (g_1, g_2)\in GL(2, \C)^2 \mid \det g_1 = \det g_2 \} 
\hookrightarrow GSpin_5(\C) = GSp_4(\C).
$$
It is known that the functorial lift corresponding to the 
embedding 
$$
SO_4(\C) \hookrightarrow SO_5(\C)
$$
is realized via the theta correspondence. 
Our Eulerian integral suggests that the Fourier-Jacobi coefficient 
of the double residue of our 
Eisenstein series provides a kernel for the theta 
correspondence for similitude groups.  
This is particularly intriguing since the Fourier-Jacobi coefficient 
construction extends directly to any group
of type $D_{3n},$ whereas there seems to be no 
hope of extending the theta correspondence to any representations of such groups which do not factor through the orthogonal quotient in any direct way.

The integral corresponding to $n=3$ provides some more 
evidence for a connection with the theta correspondence, in that the global integral unfolds to a period of $GSp_6$ 
which is known to be nonvanishing precisely on the image of 
the theta lift from $GSO_6$  (\cite{GRS-PeriodsPoles}).  

We now describe the contents of this paper. In 
section \ref{sec: notation} we fix notation and 
describe a family of global integrals, indexed 
by positive integers $n.$ 
In section \ref{sec: global} we unfold the global 
integral corresponding to the case $n=2,$ 
obtaining a global integral
involving the Whittaker function of the cusp form 
involve which, formally, factors as a product 
of local zeta integrals.
These local zeta integrals are studied in sections
\ref{sec: local zeta integrals},
\ref{section: unramified calculation},
\ref{section: local zeta integrals ii},
after certain algebraic results required
for the unramified case are established in 
section  \ref{sec: prep for unram}.
Once the local zeta integrals have been 
studied we return to the global setting for sections
\ref{section: global identity} and \ref{section: application},
where we record the global identity relating 
the original zeta integral and $
L^S( s_1, \wt \Pi \times \tau_1) L^S( s_2, \wt\Pi \times \tau_2),
$, and deduce a new identity relating poles 
of these $L$ functions and periods. 
Finally, in section \ref{section: GSO18}, we briefly describe 
what happens in the case $n=3,$ omitting details. 
We remark that the case $n=1$ is somewhat degenerate, 
as the split form of $GSO_2$ is a torus;  
our global integral appears to vanish identically in this case.

This work was undertaken while the authors were  
visiting ICERM for a program on 
Automorphic Forms, 
Combinatorial Representation Theory and Multiple Dirichlet 
Series. They thank ICERM and the organizers.

\section{Notation}\label{sec: notation}
Write $J_n$ for the matrix 
$$
\begin{pmatrix} &&1 \\ &\iddots &\\ 1&&\end{pmatrix}.
$$
If $g$ is an $n\times m$ matrix, write 
 $^tg$ for the transpose of  $g$ and 
$_tg$ for the ``other transpose,'' $J_m{}\,^tgJ_n.$ 
Let $g^* = \,_tg^{-1}.$
Let $G=GSO_{n}$ denote the identity component of 
$GO_{n}:= \{ g \in GL_{n}:  gJ_{n} \,^tg \in GL_1\cdot  J_{n}\}.$
If $n$ is odd, then $GO_n$ is the product 
of $SO_n$ and the center of $GL_n.$  
If $n$ is even, then $GSO_n$ 
is the semidirect product of $SO_{n}$ and 
$\{ \diag( \lambda I_{\frac n2}, I_{\frac n2}): \lambda \in GL_1\}.$  
Here $I_k$ is the $k \times k$ identity matrix.  
The group $GSO_{n}$ 
has a rational character $\lambda: GSO_{n} \to GL_1,$
called the similitude factor, such that 
$$
gJ_{n} \,^tg  = \lambda(g) \cdot J_{n}, 
\qquad ( g \in GSO_n).
$$
The set of upper triangular (resp. diagonal) elements of $GSO_{n}$ is a Borel subgroup (resp. split maximal 
torus) which we denote $B_{GSO_n}$ (resp. $T_{GSO_n}$).  
  A parabolic (resp. Levi) subgroup will be said to be standard if it 
contains $B_{GSO_n}$ (resp. $T_{GSO_n}$).   
The unipotent radical of $B_{GSO_n}$ will be 
denoted $U.$ We number the simple (relative to $B_{GSO_{2n}}$) roots of $T_{GSO_{2n}}$ in 
$G$ $\alpha_1, \dots, \alpha_n$ so that 
$t^{\alpha_i} = t_{ii}/t_{i+1,i+1}$ for $1 \le i \le n-1,$
and $t^{\alpha_n} = t_{n-1,n-1}/t_{n+1,n+1}.$
Here, we have 
used the exponential notation for rational characters, i.e., written $t^\alpha$
instead of $\alpha(t)$
for the value of the root $\alpha$ on the torus element $t.$

Define $m_P:GL_{3}\times GL_{3(n-1)}\times GL_{1}$
into $GSO_{6n}$ by 
\begin{equation}
\label{Eq:ParametrizationOfMP}
 m_P (g_{1},g_{2},\lambda)\mapsto diag(\lambda g_{1}, \lambda g_{2},g_{2}^{*}, g_{1}^{*}). 
\end{equation}
Denote the image by $M_P.$
It is a standard Levi subgroup.  Let $P$ be 
the corresponding standard parabolic subgroup.  
Thus, 
 $P=M_{P}\ltimes U_{P},$ where $U_P$ is the unipotent radical.    We use \eqref{Eq:ParametrizationOfMP} to identify $M_P$ with $GL_{3}\times GL_{3n-3}\times GL_{1}.$

Recall that 
a character of $F^\times \bs \A^\times$ 
(i.e., a character of $\A^\times$
trivial on $F^\times$)
is normalized if it is trivial on the positive 
real numbers (embedded into $\A^\times$ diagonally at the finite places).  An arbitrary quasicharacter
 of $F^\times \bs \A^\times$ may be expressed 
 uniquely as the product of a normalized character
 and a complex power of the absolute value.  
  If $\chi = (\chi_{1},\chi_{2},\chi_{3})$ is a triple of  normalized characters of $F^\times \bs \A^\times$
and $ s = (s_1, s_2 , s_3) \in \C^3,$  
write $(\chi; s)$ for the quasicharacter
$M_P(\A) \to \C$ by 
\begin{equation}\label{ def: (chi; s)}
(\chi;s)(g_{1},g_{2},\lambda):=\chi_1(\det g_{1})|\det g_1|^{s_1}\chi_2(\det g_{2})|\det g_2|^{s_2}\chi_3(\lambda)|\lambda|^{s_3}.
\end{equation}
Then $(\chi;s)(aI_{6n})= (\chi;s)(m_P(a^{-1}I_3, a^{-1}I_{3(n-1)}, a^2))
= \chi_1^{-3}\chi_2^{3-3n}\chi_3^2(a)|a|^{2s_3-3s_1+(3-3n)s_2}.$
The pullback of $(\chi; s)$ to a quasicharacter
of $P(\A)$ will also be denoted $(\chi; s).$

Consider the family of induced 
representations $Ind_{P(\A)}^{G(\A)}(\chi;s)$ (non-normalized induction), for fixed 
$\chi$ and $s$ varying.  
Here, we fix a maximal 
compact subgroup $K$ of $G(\A)$ and consider
$K$-finite vectors.  
The map $Ind_{P(\A)}^{G(\A)}(\chi;s)\mapsto s$ 
gives this family the structure of a fiber bundle 
over $\C^3.$  By a section we mean a function $\C^3 \times G(\A)\to \C,$ written $(s, g) \mapsto f_{\chi; s}(g),$
such that $f_{\chi;s } \in Ind_{P(\A)}^{G(\A)} (\chi; s)$
for each $s \in \C^3.$ 
A section $f_{\chi; s}$ is flat if the restriction 
of $f_{\chi; s}$ to $K$ is independent of $s.$ 
Write $\on{Flat}(\chi)$ for the space of 
flat sections.

For such a function $f_{\chi;s},$ let
$E(f_{\chi;s},g)$ be the corresponding Eisenstein series, defined
by
\begin{equation*}
  E(f_{\chi;s},g)=\sum_{\gamma\in P(F)\backslash G(F)}f_{\chi;s}(\gamma g)
\end{equation*}
when this sum is convergent and by meromorphic 
continuation elsewhere.  The sum is convergent 
for $\Re(s_1-s_2)$ and $\Re(s_2)$ both sufficiently 
large. (Cf. \cite{MW-EisensteinSeriesBook}, \S II.1.5.)

Let $Q=M_{Q}\ltimes U_{Q}$  
be the unique standard parabolic subgroup of $G$,
such that  $M_{Q}\cong GL_{2n}\times GSO_{2n}.$ 
We identify  $M_Q$ with $GL_{2n}\times GSO_{2n}$ via 
the isomorphism
\begin{equation}\label{Eq:CoordsForMQ}
  m_Q(g_{1},g_{2}):= diag(\lambda(g_2) g_{1}, g_{2}, g_{1}^{*}), \qquad (g_1 \in GL_{2n}, \; g_2 \in GSO_{2n}).
\end{equation}
The unipotent radical, $U_Q,$ of $Q$
can be described as 
$$
\left \{ 
\bpm I_{2n}&X&Y&Z'\\ 0&I_n&0&-\lt Y\\ 0&0&I_n&-\lt X\\
0&0&0&I_{2n} \epm:
Z' + X\lt Y + Y\lt X + \lt Z' = 0
\right \}.
$$
Let $\mnsw{2n}: = \{ Z \in \Mat_{{2n}\times {2n}}: \lt Z = -Z\}.$
Then we can  define a bijection (which is not a homomorphism)
$u_Q: \Mat_{{2n}\times n}\times \Mat_{{2n}\times n} \times \mnsw{2n} 
\to U_Q$ by 
$$
u_Q(X,Y,Z) = 
\bpm I_{2n}&X&Y&Z-\frac 12(X\lt Y + Y \lt X)\\ 0&I_n&0&-\lt Y\\ 0&0&I_n&-\lt X\\
0&0&0&I_{2n} \epm,
\qquad X, Y \in \Mat_{2n\times n}, \; Z \in \mnsw{2n}.
$$  
Then
$$
u_Q(X,Y,Z)u_Q(U,V,W)
= u_Q(X+U,Y+V, Z+W - \la X,V\ra + \la U,Y\ra ),
$$
$(X,Y,U,V \in \Mat_{n\times {2n}}, Z,W \in \mnsw{2n}),$
where $$
\la A,B \ra := A\lt B- B \lt A, \qquad (A,B \in \Mat_{n\times {2n}}).
$$
It follows that 
$u_Q(X,Y,Z)^{-1} = u_Q(-X,-Y,-Z),$
and that if $[x,y]=xyx^{-1}y^{-1}$ denotes the commutator, then 
$
\left[ u_Q(X,0,0), \; \, u_Q(0,Y,0)\right] 
= u_Q( 0,0, -X\lt Y +Y\lt X)=u_Q(0,0,\la Y,X\ra ).
$
Define $l(Z) = Tr( Z\cdot diag(I_n,0))=\sum_{i=1}^n Z_{i,i}$. 
For $n \in \Z$ define 
$\c H_{2n+1}$ to be $\G_a^n \times \G_a^n \times \G_a$
equipped with the product
$$(x_1, y_1, z_1) (x_2, y_2, z_2) = 
(x_1+x_2, y_1+y_2, z_1+z_2+x_1 \lt y_2 - y_1 \lt x_2).$$
Write $r$ for the map from $\Mat_{2n \times n}$
to row vectors corresponding 
to unwinding the rows: 
$r(X) := x_{1,1},\dots x_{1,n},x_{2,1},
  \dots x_{2n,n},$ and write $r'$ 
 for the similar map which 
 unwinds the rows and negates the last 
 $n.$  Explicitly: 
 $$r' (Y) = r\bspm Y_1 \\ - Y_2 \espm =   (y_{1,1},\dots y_{1,n},y_{1,1},
  \dots y_{n,n},-y_{n+1,1},\dots, -y_{2n,n})$$
  for $Y = \bspm Y_1 \\ Y_2 \espm \in \Mat_{2n\times n}, \; Y_1, Y_2 \in \Mat_{n\times n}.$
Then we can define a homomorphism from $U_{Q}$ to $\mathcal{H}_{4n^2+1}$, the Heisenberg group with $4n^2+1$ variables, by
\begin{equation*}
  j(u_Q(X,Y,Z))=(r(X), r'(Y)
   ,l(Z)).
\end{equation*}
The stabilizer of $l$ in $M_{Q}$ is $$C_{Q}:= (GSp_{2n}\times GSO_{2n})^{\circ}=\lbrace(g_{1},g_{2})\in GSp_{2n}\times GSO_{2n}\mid \lambda(g_{1})=\lambda(g_{2})^{-1}\rbrace,$$ where $\lambda(g_{i})$ is the similitude of $g_{i}$. 
For any subgroup $H$ of $GSp_{2n}\times GSO_{2n},$ let
$H^\circ := H \cap C_Q.$ 
The kernel of $l$ is a $C_Q$-stable subgroup of $U_Q,$
and is also equal to the kernel of $j.$  
Note that $\lambda(g_2)$ is also the similitude factor of $(g_1,g_2)$ 
as an element of $GSO_{6n},$
and that the center of $C_Q$ is equal to that of $GSO_{6n}.$ 
Define $T= T_{GSO_n} \cap C_Q, B = B_{GSO_n} \cap C_Q$
and $N = U \cap C_Q.$ They are a split 
maximal torus, Borel subgroup, and maximal unipotent 
subgroup of $C_Q,$ respectively.

The group of automorphisms of $\c H_{4n^2+1}$ whose restrictions
to the center of $\c H_{4n^2+1}$ are the identity is isomorphic 
to $Sp_{4n^2}.$  Identifying the two groups defines a semidirect
product $Sp_{4n^2} \ltimes \c H_{4n^2+1}.$  
Let $R_Q=C_{Q}\ltimes U_{Q}.$  The homomorphism $j: U_Q \to \c H_{4n^2+1}$
extends to a homomorphism $R_Q \to Sp_{4n^2} \ltimes \c H_{4n^2+1}.$  
Indeed, for each $c \in C_Q,$ the automorphism of $U_Q$
defined by conjugation by $c$ preserves the kernel 
of $j,$ and therefore induces an automorphism of $\c H_{4n^2+1}.$  
Moreover, this automorphism is identity on the center of 
$\c H_{4n^2+1}$ because $c$ fixes $l.$  This induces a 
homomorphism $C_Q \to Sp_{4n^2},$ 
which we denote by the same symbol $j,$ and which 
has the defining property that 
$j(cuc^{-1}) = j(c) j(u) j(c)^{-1}$ for all $c \in C_Q$ and $u \in U_Q.$
We may then regard the two homomorphisms together as 
a single homomorphism (still denoted $j$) from 
$R_Q$ to $Sp_{4n^2} \ltimes \c H_{4n^2+1}.$

For a positive integer $M,$ identify the Siegel Levi of $Sp_{2M}$ with 
$GL_M$ via the map $\bspm g& \\ &g^* \espm 
\mapsto g.$  
It acts on $\c H_{2M+1}$ 
by $g(x,y,z)g^{-1} = (x g^{-1}, y g , z).$  
Note that for $g_1 \in GSp_{2n}$ 
and $g_2 \in GSO_{2n},$ the matrix $m_Q(g_1, g_2) \in M_Q$ maps into 
$GL_{2n^2}$ if and only if 
it normalizes $\{ u_Q(X,0,0): X \in \Mat_{2n\times n}\},$ i.e., if and only if $g_2$ is of the form 
$\bspm \lambda(g_1^{-1}) g_3 &\\ &g_3^*\espm$
for $g_3 \in GL_n.$  
Write 
\begin{equation}\label{eq: def mQ1(g1, g2)}
m_Q^1( g_1, g_2) := m_Q( g_1, \bpm \lambda(g_1)^{-1} g_2 & \\ & g_2^* \epm), \qquad ( g_1 \in GSp_{2n}, \; g_2 \in GL_n).
\end{equation}
Then 
$$m_Q^1( g_1, g_2)
u_Q(X,0,0) m_Q^1( g_1^{-1}, g_2^{-1})
=
u_Q( g_1 X g_2^{-1},0,0), \qquad 
( \forall g_1 \in GSp_{2n},\; g_2 \in GL_n),
$$
and 
$j(m_Q^1(g_1, g_2)) \in GL_{2n^2}\subset Sp_{4n^2}$ is 
the matrix satisfying 
$$
r(X) j(m_Q^1(g_1, g_2))
= r(g_1^{-1} X g_2).
$$
The determinant map $GL_{2n^2} \to GL_1$
pulls back to a rational character 
of this subgroup of $C_Q$ which we denote
by $\det.$ \label{s:def of det}
Thus
$\det(m_Q^1( g_1, g_2)) = \det g_1^{-n}\det g_2^{2n} = \lambda( g_1)^{-n^2} \det g_2^{2n}$
for $g_1 \in GSp_{2n}, g_2 \in GL_n.$
On $T\subset C_Q,$ 
the rational character $\det$ coincides with the
restriction of the 
sum of the roots of 
$T_{GSO_{6n}}$
 in $\{ u_Q(0,Y,0):Y \in \Mat_{2n\times n}\}.$

Let $\psi$ be a additive character on $F\backslash \A$ and $\psi_{l}(Z):=\psi\circ l$.
The group $\c H_{4n^2+1}(\A)$ has a unique (up to isomorphism) unitary representation,
$\omega_\psi,$ 
with central character $\psi,$ which extends to a projective 
representation of $\c H_{4n^2+1}(\A) \rtimes Sp_{4n^2}(\A)$ or a 
genuine representation of $\c H_{4n^2+1}(\A) \rtimes \Sp_{4n^2}(\A),$
where $\Sp_{4n^2}(\A)$ denotes the metaplectic double cover.

\begin{lem}
The homomorphism $j:C_Q(\A) \to Sp_{4n^2}(\A)$ 
lifts to a homomorphism $C_Q(\A) \to \Sp_{4n^2}(\A).$
\end{lem}
\begin{proof}
Write $\on{pr}$ for the canonical 
projection $\Sp_{4n^2}(\A)\to Sp_{4n^2}(\A).$    
We must show that the exact sequence
$$
1 \to \{ \pm 1\} \to \on{pr}^{-1}(j(C_Q(\A))) \to j(C_Q(\A)) \to 1
$$
splits, i.e., that the cocycle determined by any choice
of section is a coboundary.
The analogous result for $Sp_{2n} \times SO_{2n},$ 
over a local field is proved in 
\cite{kudla-notes}, corollary 3.3, p. 36,
or \cite{prasad-survey}, lemma 4.4, p. 12.
The extension to $C_Q$ follows from section 5.1 of 
\cite{Harris-Kudla}. The global statement then follows from 
the corresponding local ones.
\end{proof}

Thus we obtain a homomorphism 
$R_Q(\A) \to \Sp_{4n^2}(\A) \ltimes \c H_{4n^2+1}(\A)$ 
which we still denote $j.$  
Pulling $\omega_\psi$ back through $j$ produces a 
representation of $R_Q(\A)$ which we denote $\omega_{\psi,l}.$  
This representation can be realized on the space of Schwartz 
functions on $\Mat_{{2n}\times n}(\A)$ with action by 
$$
 [\omega_{\psi,l}(u_Q(0,0,Z)).\phi] = \psi_l(Z) \phi
\qquad  [\omega_{\psi,l}(u_Q(X,0,0)).\phi](\xi) = \phi( \xi+X)$$
\begin{equation}\label{eq: Weil rep fmlas}
  [\omega_{\psi,l}(u_Q(0,Y,0)).\phi](\xi) = 
\psi_l ( \la Y, \xi \ra )\phi(\xi)=
  \psi_l( Y  \,_t \xi-\xi \lt Y)\phi(\xi),\end{equation}
$$  [\omega_{\psi, l}( m_Q^1( g_1, g_2)).\phi](\xi)
  = \gm_{\psi, \det m_Q^1( g_1, g_2)} 
  |\det m_Q^1( g_1, g_2)|^{\frac 12}
  \phi( g_1^{-1} \xi g_2 ).
$$
(Cf. \cite{GRS-Book}, p. 8.)
Here $\gm_{\psi, a}$ denotes the Weil index. 
The representation $\omega_\psi$ has an automorphic 
realization via theta functions
$$
\theta(\phi, u \wt g):= \sum_{\xi \in \Mat_{{2n}\times n}(F)} [\omega_\psi( u \wt g).\phi](\xi),
\qquad (u \in \c H_{4n^2+1}(\A),\  \wt g \in \Sp_{4n^2}(\A)).
$$
Here $\phi \in \c S(\Mat_{{2n}\times n}(\A))$ (the Schwartz space of $\Mat_{{2n}\times n}(\A)$), $\c H_{4n^2+1}(\A)$ is identified with the quotient 
of $U_Q$ by the kernel of $l,$ and $Sp_{4n^2}$ is identified 
with the subgroup of its automorphism group consisting 
of all elements which act trivially on the center.

\begin{lem}\label{Heis Lemma}
Consider the Weil representation of $\c H_{2n+1}(\A)
 \rtimes \Sp_{2n}(\A)$
and its automorphic realization by theta functions.  Let 
$V$ be a subgroup of $\c H_{2n+1}$ which intersects
the center $Z$ trivially.  Thus $V$ corresponds to an 
isotropic subspace of the symplectic space $\c H_{2n+1}/Z.$ 
Let $V^\perp = \{ v' \in \c H_{2n+1}/Z: \la v', V\ra = 0\}
\supset V,$ and let $P_V$ be the parabolic subgroup 
of $Sp_{2n}$ which preserves the flag $0 \subset V \subset V^\perp
\subset \c H_{2n+1}/Z.$  
Note that the Levi quotient of $P_V$ is canonically isomorphic 
to $GL(V) \times Sp( V^\perp/V).$ Thus $P_V$ has a projection 
onto the group $GL_1 \times Sp( V^\perp/V)$ induced by 
the canonical map onto the Levi quotient and the 
determinant map $\det: GL(V) \to GL_1.$   The function 
$$
\wt g \mapsto \int_{[V]} \theta(\phi; v\wt g) \, dv 
$$
is invariant by the $\A$-points of the kernel of this map 
on the left.
(Throughout this paper, if $H$ is an algebraic group defined
over a global field $F,$ then $[H]:= H(F) \bs H(\A).$)
\end{lem}
\begin{proof}
First assume that $V$ is the span of the last $k$ standard 
basis vectors for some $k \le n.$  Then 
$$
\int_{(F\bs \A)^k} \sum_{\xi \in F^n}[\omega_\psi
( (0, \dots, 0, v , 0)\wt g). \phi](\xi) \, dv
= \sum_{ \xi' \in F^{n-k}}
[\omega_\psi
( \wt g). \phi](0,\dots, 0, \xi'),$$
and invariance follows easily from the explicit formulae for $\omega_\psi$  given, for example on p. 8 of \cite{GRS-Book}.  
The general case follows from this special case, since any isotropic
subspace can be mapped to the span of the last $k$ standard 
basis vectors, for the appropriate value of $k,$ by using an 
element of $Sp_{2n}(F).$ 
\end{proof}

For $f_{\chi;s} \in \Ind_{P(\A)}^{G(\A)} (\chi;s),$ 
  and $\phi \in \c S(\Mat_{{2n}\times n}(\A)),$ let
\begin{equation}\label{eq: FC of E}
  E^{\theta(\phi)}(f_{\chi;s},g)=\int_{[U_{Q}]}\ du\  E(f_{\chi;s},ug)\theta(\phi,j(ug)), \qquad (g\in {C}_{Q}(\A)).
\end{equation}
Recall that $C_Q$ was identified above with a subgroup 
of $GSp_{2n} \times GSO_{2n}.$  If $g \in C_Q$ then 
$g_1$ will denote its $GSp_{2n}$ component and 
$g_2$ will denote its $GSO_{2n}$ component.
Now take two 
characters $\omega_1, \omega_2: F^\times \bs \A^\times \to \C^\times,$
and two 
cuspforms
 $\varphi_{1},$ defined on ${GSp}_{2n}(\A)$ and $\vph_1,$ defined on 
 $GSO_{2n}(\A),$
 such that $\vph_i(a\cdot g) = \omega_i(a) \vph_i(g),$ for 
 $i =1$ or $2,$ $a \in \A^\times,$ and $g \in GSp_{2n}(\A)$ or $GSO_{2n}(\A)$
 as appropriate. 
Choose $\chi_1, \chi_2, \chi_3$ so that 
$\chi_1^{-3} \chi_2^{-3} \chi_3^2 \omega_1^{-1} \omega_2$ 
is trivial, 
and consider
\begin{equation}\label{eq: global integral definition}
  I(f_{\chi;s},\varphi_{1},\varphi_{2}, \phi)=\int_{Z(\A)C_{Q}(F)\backslash C_{Q}(\A)}E^{\theta(\phi)}(f_{\chi;s},g)\varphi_{1}(g_{1})\varphi(g_{2})dg.
\end{equation}
To simplify the notation, we may also treat the 
product $\vph_1\vph_2$ as a single cuspform defined 
on the group $C_Q,$ and write 
$\vph(g) = \vph_1(g_1)\vph_2(g_2),$ and $I(f_{\chi;s}, \vph, \phi),$ etc.
Note that the integral converges absolutely and uniformly 
as $s$ varies in a compact set, simply because $E^{\theta(\phi)}(f_{\chi;s})$ 
is of moderate growth, while $\vph_1$ and $\vph_2$ are of rapid decay.

\section{Global Integral for $GSO_{12}$}
\label{sec: global}
In this section we consider a global integral \eqref{eq: global integral definition} in the case $n=2.$
Thus $G=GSO_{12}.$
If $u_Q(X,Y,Z)$ is an element of $U_Q,$ we fix individual coordinates as follows:
\begin{equation}\label{eq:X Y coords}
X=
  \begin{pmatrix}
    x_{1}& x_{2} \\ x_{3} &  x_{4} \\ x_{5}  & x_{6}\\ x_{7} & x_{8}
  \end{pmatrix},
  \qquad
  Y=
  \begin{pmatrix}
    y_{8}& y_{7} \\ y_{6} &  y_{5} \\ y_{4}  & y_{3}\\ y_{2} & y_{1}
  \end{pmatrix}\qquad
Z= \bpm 
z_1&z_2& z_3&0\\ z_4&z_5&0& -z_3\\ z_6&0&-z_5&-z_2\\ 0&-z_6&-z_4&-z_1 
\epm.
  \end{equation}

\begin{thm}\label{thm: unfolding}
For $n$ in the maximal unipotent subgroup $N$ let 
$\psi_N(n) = \psi(n_{12}+n_{23}-n_{56}+n_{57}),$
and let 
\begin{equation}\label{globalWhittakerIntegral}
W_\vph(g) = \int_{[N]}\vph(ng) \psi_N(n)\, dn.
\end{equation}
Let $U_4$ be the codimension 
one subgroup of $N$ defined by the condition $n_{23}=n_{56}.$ 
For $\phi \in \c S(\Mat_{4\times 2}(\A)),\ g \in R_Q(\A),$ write 
\begin{equation}\label{I_0(phi, g) def}
I_0(\phi, g) = 
\int_{\A^{2}}\ud a \ud b [\omega_\psi(g)\phi]\begin{pmatrix}
    a & b \\ 1 & 0 \\ 0 & 1 \\ 0 & 0
  \end{pmatrix}\psi(-a).
\end{equation}
Finally, let $w$ be the permutation matrix attached to the 
permutation 
\begin{equation}
\label{w as a permutation}\bpm
1&2&3&4&5&6&7&8&9&10&11&12\\7&10&11&12&4&5&8&9&1&2&3&6 
\epm,
\end{equation}
and let $U^w_Q = U_Q \cap w^{-1} Pw.$
Then the global integral \eqref{eq: global integral definition}
is equal to 
\begin{equation}\label{eq: global unfolded}
\int_{Z(\A)U_4(\A)\bs C_Q(\A)} W_\vph(g) \int_{U_Q^w(\A) \bs U_Q(\A)}
f_{\chi;s}(wug)I_0(\phi, ug)\, du \, dg.
\end{equation}
\end{thm}
\begin{rem}The permutation matrix $w$ represents an element 
of the Weyl group of $G$ relative to $T_G.$ 
We also record an expression for $w$ as reduced product 
of simple reflections. 
 We also introduce some notation 
for elements of the Weyl group.  
We write $w[i]$ for the simple reflection attached to the 
simple root $\alpha_i,$ and 
$w[i_1i_2\dots i_k]$ for the product $w[i_1]w[i_2]\dots w[i_k].$ 
Then $w=w[64321465432465434654].$
\end{rem}

Before proceeding to the 
proof, we need to know the structure of the set $P\backslash G \slash R_Q.$
\subsection{Description of the double coset space $P\backslash G \slash R_Q$}
\label{section:  description}
Clearly, the identity map $G \to G$ induces a map $\pr: P \bs G / R_Q
\to 
P \bs G/ Q.$  
Each element of $P \bs G/ Q$ contains a unique element of 
the Weyl group which is of minimal length.  
Recall that the group of permutation matrices which are 
contained in $G$ maps isomorphically onto 
the Weyl group of $G.$
A Weyl element of minimal length in its $P,Q$ double coset corresponds
to a permutation $\sg: \{ 1, \dots, 12\} \to \{ 1, \dots, 12\}$
such that 
\begin{itemize}
\item $\sg(13-i) = 13-\sg(i) , \forall i, $
\item $\sg$ is an even permutation.
\item If $ 1\le i<j \le 4, \; 5\le i<j \le 8,$ or $9 \le i < j \le 12,$
and if $\{ i,j\} \ne \{ 6,7\},$
then $\sg(i) < \sg(j).$
\item If $1 \le i < j \le 3,\;4 \le i < j \le 6,\;7 \le i < j \le 9,$
or $10 \le i < j \le 12,$ then $\sg^{-1}(i) < \sg^{-1}(j).$
\end{itemize}
Such a permutation $\sg$ is determined by the quadruple 
$$
( \#(\{ 1,2,3,4\}\cap \sg^{-1}(\{ 3i-2, 3i-1, 3i\})))_{i=1}^4.
$$
Deleting any zeros in this tuple gives the ordered partition of $4$ corresponding to the 
standard parabolic subgroup $P_\sg:=GL_4 \cap \sg^{-1}P \sg.$ 
(Here we identify the permutation $\sg$ with the 
corresponding permutation matrix, which is in $GSO_{12},$ and 
identify $g \in GL_4$ with $\diag(g, I_4, g^*) \in GSO_{12}.$)
Now, for any parabolic subgroup $P_o$ of $GSO_4,$ we have
$GSO_4 = P_oSO_4.$  It follows that 
$g \mapsto \sg \diag(g, I_4, g^*)$ induces a bijection 
$P_\sg \bs GL_4 / GSp_4 \leftrightarrow 
\pr^{-1}(P \cdot  \sg \cdot Q) \subset P \bs G/ R_Q.$
Therefore we must study the space $P' \bs GL_4 / GSp_4,$ 
where $P'$ is an arbitrary parabolic subgroup of $GL_4.$

\begin{lem}
Let $S$ be a subset of the set of simple roots 
in the root system of type $A_3.$ 
Let $P_S, P'_S$ denote the standard parabolic subgroups
 $GL_4,$ and $SO_6,$
respectively, corresponding to $S.$  Then 
$P_S \bs GL_4/GSp_4$ and $P'_S \bs SO_6 / SO_5$ are 
in canonical bijection.  
\end{lem}
\begin{proof}
This follows from considering the coverings 
of $SO_6$ and $GL_4$ by the group $GSpin_6$ which 
are described in \cite{Hundley-Sayag} and 
section 2.3 of \cite{Asgari-Raghuram}, respectively.
The preimage of $SO_5$ in $GSpin_6$ is $GSpin_5 = GSp_4.$ 
Since the kernels of both projections are contained in the 
central torus of $GSpin_6,$ which is contained in any parabolic subgroup
of $GSpin_6$ it follows that both $P_S \bs GL_4/GSp_4$ and $P'_S \bs SO_6 / SO_5$ are 
in canonical bijection
with $P_S'' \bs GSpin_6 / GSpin_5,$
where $P_S''$ is the parabolic subgroup of $GSpin_6$ determined
by $S.$
\end{proof}

Now, in considering $SO_6/SO_5,$ 
we embed $SO_5$ into $SO_6$ as the stabilizer of  
a fixed anisotropic element 
$v_0$ 
of the standard representation of $SO_6.$
Then 
$P_S'' \bs SO_6/ SO_5$ may be 
identified with 
 the set of $P_S'$-orbits 
in $SO_6\cdot v_0.$ 
For concreteness, take $SO_6$ to be defined using the 
quadratic form associated to the matrix $J_6,$ and take 
$v_0 = \, ^t[0,0,1,1,0,0].$  The $SO_6$ orbit of $v_0$ 
is the set of vectors satistfying $^tv\cdot J_6 \cdot v = {}\,^t\!
v_0 \cdot J_6 \cdot v = 2.$
Note that each of the permutation matrices representing
a simple reflection attached to an outer 
node in the Dynkin diagram maps $v_0$ to $v_1:=\,^t[0,1,0,0,1,0],$
and that a permutation matrix representing the simple reflection attached to the middle node 
of the Dynkin diagram maps $v_1$ to $v_2:=\,^t[1,0,0,0,0,1].$ 

\begin{lem}\label{lem: P SO6 SO5 reps}
Number the roots of $SO_6$ so that $\alpha_2$ is the 
middle root.  (This is not the standard numbering 
for $SO_6,$ but it matches the standard numbering 
for $GL_4,$ and the numbering inherited as a subgroup 
of $GSO_{12}$.)
Write $V$ for the standard representation of $SO_6.$ 
The decomposition of $SO_6 \cdot v_0$ into $P_S'$ orbits is as follows: 
$$
\begin{array}{|c|c|c|}\hline
S &
\text{orbit reps in }V &
\text{double coset reps} 
\\\hline
\emptyset & v_0, v_1, v_2& e, w[1], w[2]w[1]\\ 
\{1\}, \{ 3\}, \text{ or } \{1, 3\} & v_0, v_2
& e, w[1]
\\
\{2\}& v_0, v_1&e, w[2]w[1] \\
\{1,2\} \{ 2,3\} \text{ or } \{1,2,3\} & v_0 & e\\\hline
\end{array}
$$
\end{lem}
\begin{proof}
Direct calculation.
\end{proof}
\begin{rem}
As elements of $GSO_{12},$ 
the double coset representatives are identified with 
permutations of $\{ 1,\dots, 12\}.$  Writing these permutations
in cycle notation, we have
$w[1] =(1,2)(11,12),\; w[2]w[1]
= (1,3,2)(10,11,12).$
Replacing $w[1]$ by $w[3]$ in any  of the representatives 
above produces a different element of the same double coset. 
\end{rem}

\subsection{Proof of theorem \ref{thm: unfolding}}
\label{sec: proof of unfolding}
We now apply this description of $P \bs G / R_Q,$
to the study of $I(f_{\chi;s}, \vph, \phi).$  
For this section only, let $w_0$ be the permutation matrix 
attached to \eqref{w as a permutation}, and let $w$ be an 
arbitrary representative for $P(F)\backslash G(F)\slash R_Q(F).$ 

The global  integral \eqref{eq: global integral definition} 
is equal to
$$
  \sum_{w\in P(F)\backslash G(F)\slash R_Q(F)}
  I_w(f_{\chi;s}, \vph, \phi),$$where
    $$I_w(f_{\chi;s}, \vph, \phi)=
  \int_{Z(\A)U_{Q}^{w}(F)C_{Q}^{w}(F)\backslash C_{Q}(\A)U_{Q}(\A)} f_{\chi;s}(wug)\theta (\phi,j(ug))\varphi(g)dg,
$$
where $C_{Q}^{w}=C_{Q}\cap w^{-1}Pw$, and 
 $U_{Q}^{w}=U_Q \cap w^{-1} P w$.

\begin{pro}\label{prop:  only one double coset contributes}
If $w$ does not lie in the double coset containing
$w_0,$ 
then $I_w( f_{\chi;s}, \vph_1, \vph_2)=0.$
Consequently, $I(f_{\chi;s}, \vph_1, \vph_2) = I_{w_0}(f_{\chi;s}, \vph_1, \vph_2).$
\end{pro}
\begin{proof}
Write $w = \sg \nu$ where $w$ is a permutation 
of $\{ 1, \dots 12\}$ satisfying the four conditions listed 
at the beginning of section \ref{section:  description}, 
and $\nu$ is one of the representatives for 
$P_\sg\bs M_Q/ C_Q$ given in 
the table in lemma \ref{lem: P SO6 SO5 reps}.
The integral $I_w( f_{\chi;s}, \vph, \phi )$ vanishes
if $\psi_l$ is nontrivial on $U_Q^w:=U_Q \cap w^{-1} P w,$  
or equivalently, if the character $\nu \cdot \psi_l$ obtained by 
composing $\psi_l$ with conjugation by $\nu$
is nontrivial on $U_Q \cap \sg^{-1} P \sg.$ 
For our representatives $\nu,$ we have
$$
\nu \cdot \psi_l(u_Q(0,0,Z))
= \begin{cases}
\psi( Z_{1,9}+Z_{2,10}) , & \nu = e, \\
\psi( Z_{1,10}+Z_{2,9}), & \nu = w[1],\\
\psi(Z_{1,11}+Z_{3,9}), & \nu = w[2]w[1].
\end{cases}
$$

There are 25 possibilities for $\sg.$  
However, it's clear that $I_w( f_{\chi;s}, \vph, \phi )$ vanishes,
regardless of $\nu,$ if $\sg(1) < \sg(9),$ or if $\sg(2) < \sg(10).$
This eliminates all but seven possibilities for $\sg.$ 
For the remaining seven, the above criterion shows that 
$I_w( f_{\chi;s}, \vph, \phi )$ vanishes unless $\nu$ is trivial.

Assume now that $\psi_l$ is trivial on $U_Q^w.$
This means that the image of $U_Q^w$
in the Heisenberg group intersects the center trivially, 
and maps to an isotropic subspace of the quotient 
$\c H_{17}/Z(\c H_{17})$ (which has the 
structure of a symplectic vector group). 
Write $V$ for this subspace and 
$V^\perp$ for its perp space.  
Define  $P_V\subset Sp_{16}$ 
as in lemma \ref{Heis Lemma}, and let $P_V^1$ denote
the kernel of the canonical projection $P_V \to GL_1 \times Sp( V^\perp/V).$ 
It follows immediately from lemma \ref{Heis Lemma} and the cuspidality 
of $\vph$ that 
$I_w( f_{\chi;s}, \vph, \phi )$ vanishes
whenever
 $P_V^1 \cap C_Q$ contains the unipotent 
radical of a proper parabolic subgroup of $C_Q.$
This applies to each of the remaining double coset 
representatives, except for $w_0.$
 \end{proof}

The following lemma is useful in our calculation.
\begin{lem}
  \label{lem:1}
  Let $f_{1}$, $f_{2}$ be two continuous functions on $(F\bs \A)^{n}$, and $\psi$ a nontrivial additive character on $F\bs \A$. Then
  \begin{equation}
    \label{eq:3}
    \int_{(F\bs \A)^{n}}\ud x f_{1}(x)f_{2}(x)=\sum_{\alpha\in F^{n}}\int_{(F \bs \A)^{n}}\ud x f_{1}(x)\psi(\alpha\cdot x)\int_{(F \bs \A)^{n}}\ud y f_{2}(y)\psi^{-1}(\alpha \cdot y).
  \end{equation}
Moreover, if $\int_{(F\bs \A)^{n}}\ud x f_{1}(x)=0$, then one can replace $\sum_{\alpha\in F^{n}}$ by $\sum_{\alpha\in F^{n}-\{0\}}$ in the formula above. 
\end{lem}
\begin{proof}
By Fourier theory on $F\bs \A$,
\begin{equation*}
  f_{i}(x)=\sum_{\alpha\in F^{n}} \psi(-\alpha\cdot x)\hat{f}_{i}(\alpha),
\end{equation*}
where $\hat{f}_{i}(\alpha)=\int_{(F\bs \A)^{n}}\ud x f_{1}(x)\psi(\alpha x)$ for $i=1,2$. So the left hand side of \eqref{eq:3} is equal to
\begin{equation}
  \label{eq:5}
  \sum_{\alpha,\beta\in F^{n}}\hat{f}_{1}(\alpha)\hat{f}_{2}(\beta)\int_{(F\bs \A)^{n}}\ud x \psi(-(\alpha+\beta)\cdot x).
\end{equation}
The integral on $x$ vanishes when $\alpha+\beta\neq 0$, and equals 1 if $\alpha+\beta=0$, so \eqref{eq:5} equals
\begin{equation*}
  \sum_{\alpha\in F^{n}}\hat{f}_{1}(\alpha)\hat{f}_{2}(-\alpha),
\end{equation*}
which is the right hand side of \eqref{eq:3}. When $\int_{(F\bs \A)^{n}}\ud x f_{1}(x)=0$, we have $\hat{f}_{1}(0)=0$, so we can replace $\sum_{\alpha\in F^{n}}$ by $\sum_{\alpha\in F^{n}-0}.$
\end{proof}
From now on, let $w=w[64321465432465434654].$
Then
\begin{align}
  \label{eq:10}
  & U_{Q}^{w}:=U_Q \cap w^{-1} P w= \left\{u_{Q}^{w}(y_{7},y_{8})= u_Q\left( 0,\; \bpm y_8&y_7 \\ 0&0\\0&0\\0&0\epm,\; 0 \right)  : y_{7},y_{8}\in F\right\}, \\
  & C_{Q}^{w}:=C_Q \cap w^{-1} P w= (P_{1}\times P_{2})^{\circ} \notag
\end{align}
where $P_{1}$ is the Klingen parabolic subgroup of $GSp_{4}$  and $P_{2}$ is the Siegel parabolic subgroup of $GSO_{4}$. Let $P_{1}=M_{1}\ltimes U_{1}$ and $P_{2}=M_{2}\ltimes U_{2}$ be their Levi decompositions. 
Note that $f_{\chi;s}(wug)=f_{\chi;s}(wg)$ for all $u\in U_{Q}^{w}.$  So,
by proposition  \ref{prop:  only one double coset contributes},
$I(f_{\chi;s}, \vph, \phi)$ is equal to 
\begin{equation}\label{eq a}
\int_{Z(\A)C_Q^w(F) \bs C_Q(\A)}
\vph(g) \int_{U_Q^w(\A)\bs U(\A)} f_{\chi;s}(wu_2g)
\int_{[U_Q^w]}
\theta(\phi, j(u_1u_2g))\, du_1\, du_2\, dg. 
\end{equation}
But, for $u=u_{Q}^{w}(y_{7},y_{8})$ (defined in \eqref{eq:10}),
\begin{align}
  [\omega_\psi(j(u))\phi_1](\xi)=\phi_1(\xi)\psi(\xi_{7}y_{7}+\xi_{8}y_{8}),
 \qquad \xi = \bpm \xi_1 & \xi_2 \\ \xi_3& \xi_4 \\ \xi_5 & \xi_6 \\
 \xi_7 & \xi_8 \epm,
\end{align}
for any $\phi_1 \in \c S(\Mat_{4 \times 2}(\A)).$
It follows that \eqref{eq a} is equal to 
\begin{equation}\label{eq b}
\int_{Z(\A)C_Q^w(F) \bs C_Q(\A)}
\vph(g) \int_{U_Q^w(\A)\bs U(\A)} f_{\chi;s}(wug)
\theta_0(\phi, j(ug))\, du\, dg, 
\end{equation}
where
$$
\theta_0(\phi, u \wt g) 
:= \sum_{\xi\in \Mat_{3\times 2}(F)}
[\omega_\psi( u \wt g).\phi]\bpm \xi\\ 0 \epm,
\qquad (u \in \c H_{17}(\A),\  \wt g \in \Sp_{16}(\A))
$$

Now, $C_{Q}^{w}=(M_{1}\times M_{2})^{\circ}\ltimes (U_{1}\times U_{2})$, and $f_{\chi;s}(wu_{1}u_{2}g)=f_{\chi;s}(wg)$, for any $u_{1}\in U_{1}$, $u_{2}\in U_{2}$, and $g\in G.$
Moreover, if
\begin{equation}
  \label{eq:4}
  U_{2}(a)=
  \begin{pmatrix}
    1 & &a & \\ & 1 & &-a\\ &&1& \\ &&& 1
  \end{pmatrix},
\end{equation}
then $[\omega_\psi(U_{2}(a)u_{1}g)\phi]\begin{pmatrix}
    \xi_{1}& \xi_{2}\\ \xi_{3} & \xi_{4} \\ \xi_{5} & \xi_{6}\\ 0 & 0
  \end{pmatrix}=\psi(a(\xi_{3}\xi_{6}-\xi_{4}\xi_{5}))[\omega_\psi(u_{1}g)\phi]\begin{pmatrix}
    \xi_{1}& \xi_{2}\\ \xi_{3} & \xi_{4} \\ \xi_{5} & \xi_{6}\\ 0 & 0
  \end{pmatrix}$. 
It then follows from the cuspidality of $\vph$ that 
 \eqref{eq b} is equal to 
\begin{equation}\label{eq c}
\int_{Z(\A)C^w_Q(F) \bs C_Q(\A)}
\vph(g) \int_{U_Q^w(\A)\bs U(\A)} f_{\chi;s}(wug)
\theta_1(\phi, j(ug))\, du\, dg, 
\end{equation}
where
$$
\theta_1(\phi, u \wt g) 
:= \sum_{\xi\in \Mat_{3\times 2}(F):\ (\xi_{3}\xi_{6}-\xi_{4}\xi_{5})\ne 0}
[\omega_\psi( u \wt g).\phi]\bpm \xi\\ 0 \epm,
\qquad (u \in \c H_{17}(\A),\  \wt g \in \Sp_{16}(\A)).
$$
 The group $(M_{1}\times M_{2})^{\circ}$ is
 the set of all 
 \begin{equation}\label{m: GL2 x GL2 x GL1 -> M1 x M2 circ}
 m(g_3, g_4, t):=
 diag(t\det g_{3},g_{3},t^{-1};\det g_{3} g_{4},g_{4}^{*};t \det g_{3},\det g_{3}\cdot g_{3}^{*},t^{-1})
 \end{equation}
  where $g_{3}\in GL_{2}$, $g_{4}\in GL_{2}$ and $t\in GL_{1}$. Note that the summation over $(\xi_{1},\xi_{2})$ is invariant under the action of $(M_{1}\times M_{2})^{\circ}$. Consider the action of $(M_{1}\times M_{2})^{\circ}$ on $\lbrace(\xi_{3},\xi_{4},\xi_{5},\xi_{6})\mid \det
\begin{pmatrix}
  \xi_{3} & \xi_{4} \\ \xi_{5} & \xi_{6}
\end{pmatrix}
\neq 0\rbrace$. It is not hard to see that it is transitive, and the stabilizer of $(1,0,0,1)$ is $\lbrace m(t,g_{3},g_{4})\mid g_{4}= g_{3}\cdot \det(g_{3})^{-1}\rbrace$, which is the same as  $\lbrace M_{5}(t,g_{3})=diag(t \det g_{3},g_{3},t^{-1}; g_{3},g_{3}^{*}\det g_{3}; t \det g_{3}, g_{3}^{*}\det g_{3}, t^{-1}): g_3 \in GL_2, \ t \in GL_1\rbrace$. We denote this group by $M_{5}.$  Let $\psi_{U_{2}}$ be a character on $U_{2}$ defined by $\psi_{U_{2}}(U_{2}(a))=\psi(a)$, then  equation \eqref{eq c} is
 equal to 
\begin{equation}
\label{eq d}
\int_{Z(\A)M_5(F)U_1(F)U_2(\A)\bs C_Q(\A)}
\vph^{(U_2,\psi_{U_2})}(g) 
\int_{U_Q^w(\A) \bs U(\A)}
f_{\chi;s}(wug) 
\theta_2(j(ug))
\, du
\, dg,
\end{equation}
where
$$
\theta_2(j(ug)):=
\sum_{(\xi_1, \xi_2) \in F^2}
[\omega_\psi(j(ug)).\phi]\begin{pmatrix}
    \xi_{1}& \xi_{2}\\ 1 & 0 \\ 0 & 1\\ 0 & 0
  \end{pmatrix},
$$
and the notation $\varphi^{(U_2, \psi_{U_2})}$ is defined as follows. 
For any unipotent subgroup $V$ of 
an  $F$-group $H,$ 
character $\vartheta$
of $V,$ 
and smooth left $V(F)$-invariant function $\Phi$ on 
$H(\A),$ we define
$$
\Phi^{(V,\vartheta)}(h) : = \int_{[V]}\Phi(vh)  \vartheta(v) \, dv.
$$
Now, $U_{1}$ consists of elements
\begin{equation}
  \label{eq:6}
  U_{1}(a,b,c)=
  \begin{pmatrix}
    1 & a & b & c \\ & 1 & & b \\ & & 1 & -a \\ & & & 1
  \end{pmatrix}\in GSp_{4},
\end{equation}
and for any $g\in R_{Q}$, 
\begin{equation*}
  [\omega_{\psi}(U_{1}(0,0,c)g)\phi]\begin{pmatrix}
    \xi_{1}& \xi_{2}\\ 1 & 0 \\ 0 & 1\\ 0 & 0
  \end{pmatrix}=[\omega_{\psi}(g)\phi]\begin{pmatrix}
    \xi_{1}& \xi_{2}\\ 1 & 0 \\ 0 & 1\\ 0 & 0
  \end{pmatrix}.
\end{equation*}
Factoring the integration over $U_{1}$ and applying lemma \ref{lem:1} to functions $$(a,b)\mapsto \omega_\psi(U_{1}(a,b,0)g)\phi\begin{pmatrix}
    \xi_{1}& \xi_{2}\\ 1 & 0 \\ 0 & 1\\ 0 & 0
  \end{pmatrix}\quad \text{  and  }\quad(a,b)\mapsto \int_{F\backslash \A}\ud c \varphi_{2}(U_{1}(a,b,c)g),$$ 
  we deduce that \eqref{eq d} is equal to
\begin{equation}
\label{eq e}
\int\limits_{Z(\A)M_5(F)U_1(\A)U_2(\A)\bs C_Q(\A)}
\vph^{(U_3,\psi_{U_3}^{\alpha,\beta})}(g)
\int\limits_{U_Q^w(\A) \bs U(\A)}
f_{\chi;s}(wug) 
\theta_2^{(U_1, \psi_{U_1}^{-\alpha,-\beta})}(j(ug))\, du
\, dg,
\end{equation}
where $\psi_{U_1}^{\alpha,\beta}(U_1(a,b,c)) = 
\psi(\alpha a+ \beta b),$ $U_3=U_1U_2,$ and 
$\psi_3^{\alpha,\beta}=\psi_{U_2} \psi_{U_1}^{\alpha, \beta}.$
The group $M_{5}(F)$ acts on 
$U_1(\A)$ and permutes the nontrivial characters $\psi_{U_1}^{\alpha,\beta}$ transitively.  The stabilizer of $\psi_{U_1}:= \psi_{U_1}^{1,0}$ 
 is 
\begin{equation}
  \label{eq:9}M_6:= 
 \left\{ M_6(a_{1},a_{2},a_{4})
 \right\},\qquad \text{ where }
 M_6(a_{1},a_{2},a_{4})=
  M_{5}\left(a_{4}^{-1},\begin{pmatrix}
  a_{1} & a_{2} \\ 0 & a_{4}
\end{pmatrix}\right).
\end{equation} 
Hence equation \eqref{eq e} is equal to 
\begin{equation}
\label{eq f}
\int\limits_{Z(\A)M_6(F)U_1(\A)U_2(\A)\bs C_Q(\A)}
\vph^{(U_3,\psi_{U_3})}(g)
\int\limits_{U_Q^w(\A) \bs U(\A)}
f_{\chi;s}(wug) 
\theta_2^{(U_1, \ol\psi_{U_1})}(j(ug))\, du
\, dg,
\end{equation}
where $\psi_{U_3} = \psi_{U_1}\psi_{U_2}.$ 

Note that 
\begin{equation*}
  [\omega_\psi(U_{1}(a,b,0)g)\phi]
  \begin{pmatrix}
    \xi_{1} & \xi_{2}\\ 1 & 0 \\ 0 & 1\\ 0 & 0
  \end{pmatrix}=[\omega_\psi(g)\phi]
  \begin{pmatrix}
    \xi_{1}+a & \xi_{2}+b\\ 1 & 0 \\ 0 & 1\\ 0 & 0
  \end{pmatrix},
\end{equation*}
and that for $\xi_{1},\xi_{2}\in F$, $\psi(\alpha\cdot(a+\xi_{1})+\beta\cdot(b+\xi_{2}))=\psi(\alpha\cdot a+\beta\cdot b)$. We can combine the summation on $(\xi_{1},\xi_{2})$ with the integral over $(a,b)$. It 
follows that  $\theta_2^{(U_1, \ol\psi_{U_1})}(g)=I_0(\phi, g),$
defined in \eqref{I_0(phi, g) def}.
 Let $M_{6}= U_{6}T_{6}$ be the Levi decomposition.  It is not hard to see that both $I_0(\phi, g)$ and the function $g\mapsto f(wg)$
 are invariant on the left by $U_6(\A).$  
 So, \eqref{eq f} is equal to 
 \begin{equation}
\label{eq g}
\int\limits_{Z(\A)T_6(F)U_1(\A)U_2(\A)\bs C_Q(\A)}
\vph^{(U_4,\psi_{U_4})}(g)
\int\limits_{U_Q^w(\A) \bs U(\A)}
f_{\chi;s}(wug) 
\theta_2^{(U_1, \ol\psi_{U_1})}(j(ug))\, du
\, dg,
\end{equation}
where $U_4 = U_3U_6,$
and $\psi_{U_4}$ is the extension of $\psi_{U_3}$
to a character of $U_4$ which is trivial on $U_6.$ 
Now, 
$$
\vph^{(U_4,\psi_{U_4})}(g)
= 
\int_{F \bs \A} \vph_1^{(U_1, \psi_1)}\left(
\bpm 1&&&\\&1&r&\\&&1&\\&&&1\epm
g_1\right)
\vph_2^{(U_2,\psi_2)}\left(\bpm 1&r&&\\&1&&\\&&1&-r\\&&&1\epm
g_2\right).
$$
Let $N_1$ denote the standard maximal unipotent subgroup 
of $GSp_4$ and $N_2$ that of $GSO_4.$  Let 
$\psi_{N_1}^\gm$ and $\psi_{N_2}^\gm$ be the 
extensions of $\psi_{U_1}$ and $\psi_{U_2}$ to 
characters of $N_1(\A)$ and $N_2(\A)$ respectively, 
such that 
$$
\psi_{N_1}^\gm\bpm 1&&&\\&1&r&\\&&1&\\&&&1\epm
=\psi_{N_2}^\gm\bpm 1&r&&\\&1&&\\&&1&-r\\&&&1\epm
= \psi(\gm r).
$$
Then it follows from lemma \ref{lem:1} (and the cuspidality 
of $\vph_1, \vph_2,$)
that 
$$
\vph^{(U_4,\psi_{U_4})}(g)= \sum_{\gm \in F^\times}
\vph_1^{(N_1,\psi_{N_1}^\gm)}(g_1) 
\vph_2^{(N_2, \psi_{N_2}^{-\gm})}(g_2) = \vph^{(N, \psi_N^\gm)}(g),
$$
where $N=N_1N_2,$ a maximal unipotent subgroup of $C_Q,$
and for $\gm \in F^\times,$
$ \psi_N^\gm = \psi_{N_1}^\gm\psi_{N_2}^{-\gm}.$
We plug this in to \eqref{eq g}.  The group $T_6$ acts on 
the characters $\psi_N^\gm$ transitively, and 
the stabilizer of $\psi_N : = \psi_N^1$ is the center of $C_Q.$ 
Since $\vph^{(N, \psi_N)}(g) = W_\vph(g),$ this completes the 
proof of theorem \ref{thm: unfolding}.

\section{Preparation for the unramified calculation}
\label{sec: prep for unram}
In this section, we establish some results which describe the 
structure of the symmetric algebras of some representations
of  $Sp_4\times SL_2,$ and $Sp_4 \times SL_2 \times SL_2,$
which will be used to relate our local zeta integrals
to products of Langlands $L$-functions.

We first consider some representations of 
$Sp_4 \times SL_2.$ 
Let $\varpi_1$ and $\varpi_2$ denote the fundamental 
weights of $Sp_4$ and $\varpi$ that of $SL_2.$  
Write $V_{(n_1, n_2; m)}$ for the  irreducible $Sp_4 \times SL_2$-module with highest weight $n_1\varpi_1+n_2\varpi_2 + m\varpi,$
 and let $[n_1, n_2; m]$ denote its trace. 
 \begin{pro}
For $i,j,n_1, n_2$ and $m$ all
non-negative integers, let 
 $\mu_{i,j}(n_1,n_2;m)$ denote the multiplicity 
 of $V_{(n_1, n_2; m)}$ in $\sym^i V_{(1,0;1)}\otimes \sym^j V_{(1,0;0)}.$
 Then 
$$\begin{aligned}
&
\sum_{i,j,n_1,n_2,m=0}^\infty 
\mu_{i,j}(n_1,n_2; m) t_1^{n_1}t_2^{n_2} t_3^m x^i y^j =
\\
&\frac{1-t_1t_2t_3 x^3 y^2}{(1-t_1t_3x)(1-x^2)(1-t_2x^2)(1-t_1y)(1-t_3xy)(1-t_2t_3xy)(1-t_1x^2y)(1-t_2x^2y^2)}.\end{aligned}
$$
\end{pro}
\begin{proof}
We first describe $\sym^j V_{(1,0;1)}.$ By treating $V_{(1,0;1)}.$ 
Write $V_{(n_1,n_2)}$ for the irreducible $Sp_4$-module with 
highest weight $n_1\varpi_1+n_2\varpi_2.$ Then we may 
regard $V_{(1,0;1)}$ as two copies of $V_{(1,0)}$ with 
the standard torus of $SL_2$ acting on them by eigenvalues, 
say, $\eta$ and $\eta^{-1}.$ 
Then, using the well known fact that $\sym^kV_{(1,0)}=V_{(j,0)},$
and the decomposition of $V_{(m,0)}\otimes V_{(n,0)}$ described 
in \cite{King71}, 
$$
\Tr\sym^nV_{(1,0;1)} = \sum_{n_1 =0}^n \eta^{n-2n_1}
\sum_{\ell = 0}^{\min(n_1, n-n_1)} \sum_{j=0}^\ell
[n-2\ell,j]
= \sum_{\ell =0}^{\lfloor \frac n2 \rfloor}
\sum_{j=0}^\ell
[n-2\ell,j; n-2\ell],
$$
whence
$$
\sum_{n=0}^\infty
x^n \Tr \sym^n V_{(1,0;1)}
= \sum_{j,k,n=0}^\infty x^{n+2j+2k} [n,j;n]
= \frac{1}{1-x^2} \sum_{j,n=0}^\infty [n,j;n] x^{n+2j}.
$$
Using \cite{King71} again to compute
$V_{(n,j)}\otimes V_{(m,0)}$ one obtains
$$\begin{aligned}
\sum_{n,m,j=0}^\infty
[n,j][m,0]
t^n x^{n+2j}y^m 
&=
\frac1{1-txy}
\sum_{{n,j,i_2,k=0}\atop {n+i_2 \ge k}}^\infty 
[n+m+i_2-k_2, j+k_2] t^{n}x^{n+2j+2i_2}y^{m+i_2+k}.
\end{aligned}
$$
It follows that 
$$\begin{aligned}
\sum_{i,j,n_1,n_2,m=0}^\infty 
\mu_{i,j}(n_1,n_2; m) &t_1^{n_1}t_2^{n_2} t_3^m x^i y^j\\& =
\frac{1}{1-x^2}\frac1{1-t_3xy}
\sum_{{n_1,n_2,i_2,k,m=0}\atop {n+i_2 \ge k}}^\infty 
t_1^{n_1+m+i_2-k}t_2^{n_2+k} t_3^{n_1}x^{n_1+2n_2+2i_2}y^{m+i_2+k}\\
& =
\frac{1}{1-x^2}\frac1{1-t_3xy}\frac1{1-t_2x^2}
\frac1{1-t_1y}
\sum_{{n_1,i_2,k=0}\atop {n+i_2 \ge k}}^\infty 
t_1^{n_1+i_2-k}t_2^{k} t_3^{n_1}x^{n_1+2n_2+2i_2}y^{i_2+k},
\end{aligned}
$$
and the result then follows from the identity $$
\sum_{n_1=0}^\infty 
\sum_{n_2=0}^\infty 
\sum_{j,k=0}^{n_2} 
u^{n_2} v^k w^j 
= \frac{1-u^2 vw}{(1-u)(1-uv)(1-uw)(1-uvw)}.
$$
\end{proof}
\begin{cor}\label{cor: sym alg V x W}
For $n=(n_1, n_2, n_3)$ let $[n]=[n_1,n_2;n_3],$ and, 
let 
$$\bm 
a= {}^t\bbm1&0&1&1& 2&2& 2\ebm\\
\\
 b={}^t \bbm0&1&1&1&0&1&2\ebm\em,
\qquad
g= {}^t
\bbm
1&1&0&0&0&1&0\\
0&0&0&1&1&0&1\\
1&0&1&1&0&0&0
\ebm 
$$
Then
$
\sum_{i=0}^\infty 
x^i \Tr \sym^i V_{(1,0;1)}
\sum_{j=0}^\infty 
y^j \Tr \sym^j V_{(1,0;0)}$
equals $$\frac1{1-x^2}
\left[
\sum_n 
[n\cdot g] x^{n \cdot a } y^{n \cdot b}
-
\sum_n 
[n\cdot g+(1,1,1)] x^{n \cdot a +3} y^{n \cdot b+2}\right],
$$
where $n$ is summed over row vectors $n=(n_1, \dots, n_7)\in \Z_{\ge 0}^7.$
\end{cor}
Our next result 
describes the decomposition of $\sym^i V_{(1,0;1)} \otimes \sym^j V_{(1,0;0)} \otimes \sym^k V_{(1,0;0)}.$
It is an identity of rational functions in $6$ variables. To keep the 
notation short, we often reflect dependence only on arguments which will vary.
Let
$$\begin{aligned}
d(t_1, t_2)&= 
(1-t_1t_3x)(1-t_2x^2)(1-t_1y)(1-t_3xy)(1-t_2t_3xy)(1-t_1x^2y)(1-t_2x^2y^2)\\
& = \left[\sum_{n \in \Z_{\ge 0}^7 }
t_1^{n\cdot g_1}t_2^{n\cdot g_2}t_3^{n\cdot g_3} x^{n \cdot a} y^{n\cdot b}\right]^{-1},
\end{aligned}
$$
where $g_1, g_2$ and $g_3$ are the three columns of the matrix $g$
in corollary \ref{cor: sym alg V x W}.
Define rational functions $\gm_1, \dots , \gm_7$ by
$$
\begin{aligned}
\sum_{n_1=0}^{N_1}
\sum_{n_2=0}^{N_2}
\sum_{k=0}^{n_1+n_2} \!\!
u^{n_1} v^{n_2} w^k
& = \gm_1(u,v,w) + \gm_2(u,v,w) u^{N_1}
+ \gm_3(u,v,w)v^{N_2}
+\gm_4(u,v,w)u^{N_1}v^{N_2}\\&
+\gm_5(u,v,w)(uw)^{N_1}+\gm_6(u,v,w) (vw)^{N_2}
+\gm_7(u,v,w)(uw)^{N_1}(vw)^{N_2},
\end{aligned},
$$
and let 
$c_i =\gm_i(t_1/z,  t_1 z/t_2, t_2z/t_1).$
\begin{pro}\label{pro: mu ijk n1 n2 m}
Let $\mu_{i,j,k}(n_1, n_2; m)$
denote the multiplicity 
of $V_{(n_1, n_2; m)}$ in 
$\sym^i V_{(1,0;1)} \otimes \sym^j V_{(1,0;0)} \otimes \sym^k V_{(1,0;0)}.$
Then 
$$\begin{aligned}
&
\sum_{i,j,k, n_1,n_2,m=0}^\infty 
\mu_{i,j,k}(n_1, n_2; m) t_1^{n_1} t_2^{n_2} t_3^m 
x^i y^j z^k = (1-x^2)^{-1}(1-t_1z)^{-1}\times \\
&\times
\frac{c_1
\nu( t_2z)}
{d(z,t_2)}
+\frac{c_2
\nu( t_1t_2)}
{d(t_1,t_2)}
+\frac{c_3
\nu( t_1z^2)}
{d(z,t_1z)}
+\frac{c_4
\nu( t_1^2z)}
{d(t_1,t_1z,)}
+\frac{c_5
\nu( t_2^2z)}
{d(t_2 z, t_2)}
+\frac{c_6
\nu( t_2z^3)}
{d(z, t_2z^2)}
+\frac{c_7
\nu( t_2^2z^3)}
{d(t_2 z, t_2z^2)},
\end{aligned}$$
where 
$c_1, \dots, c_7$ and $d$ are as above and  let $\nu(u) 
= 1-ut_3x^3y^2.$ 
\end{pro}
\begin{proof}
From \cite{King71} again, one deduces that
$$
[m_1,m_2]\cdot \sum_{\ell=0}^\infty 
[\ell,0]x^\ell
=
\sum_{\ell=0}^\infty
\sum_{i_1=0}^{m_1}
\sum_{i_2=0}^{m_2}
\sum_{k=0}^{i_1+i_2}
[i_1+i_2-k+\ell, m_2-i_2+k] x^{\ell+m_1-i_1+i_2+k}.
$$
Combining with corollary \ref{cor: sym alg V x W} 
gives $$\begin{aligned}
\sum_{i,j,k,n_1,n_2,m=0}^\infty 
\mu_{i,j,k}(n_1, n_2; m) t_1^{n_1} t_2^{n_2} t_3^m 
x^i &y^j z^k 
=\\
\frac{1}{1-x^2}\frac1{1-t_1z}
\sum_{n\in \Z_{\ge 0}^7}
x^{n\cdot a} y^{n\cdot b}
t_3^{n\cdot g_3}&\left(
\sum_{i_1=0}^{n\cdot g_1}
\sum_{i_2=0}^{n\cdot g_2}
\sum_{k=0}^{i_1+i_2}
t_1^{i_1+i_2-k}
t_2^{n\cdot g_2-i_2+k}
z^{n\cdot g_1 -i_1+i_2+k}\right.\\
&\left.
-
x^3y^2t_3
\sum_{i_1=0}^{n\cdot g_1+1}
\sum_{i_2=0}^{n\cdot g_2+1}
\sum_{k=0}^{i_1+i_2}
t_1^{i_1+i_2-k}
t_2^{n\cdot g_2+1-i_2+k}
z^{n\cdot g_1+1 -i_1+i_2+k}
\right),
\end{aligned}
$$

and   
and simplifying this rational function gives the result.
\end{proof}

\begin{pro}\label{pro: decomp of V times V'}
Let $V_{(n_1, n_2; n_3; n_4)}$ denote the irreducible 
representation of $Sp_4 \times SL_2 \times SL_2$ 
such that $Sp_4$ acts with highest weight $n_1\varpi_1+n_2\varpi_2,$
the first $SL_2$ acts with highest weight $n_3,$
and the second $SL_2$ acts with highest weight $n_4.$ 
For $n=(n_1, n_2; n_3;n_4),$ let $\mu_{i,j}(n)$ denote the 
multiplicity of $V_n$ in $\sym^i V_{(1,0;1;0)} \otimes \sym^j V_{(1,0;0;1)}.$ 
Then 
$$\sum_{i,j=0}^\infty
\sum_{n \in \Z_{\ge 0}^4} \mu_{i,j}(n) t_1^{n_1}t_2^{n_2}t_3^{n_3}t_4^{n_4} x^i y^j
= \frac{\nu(x,y,t)}{\delta(x,y,t)},
$$
where 
$$\begin{aligned}
\nu(x,y,t)=&1 - t_1 t_2 t_3 t_4^2 x^3 y^2 - t_1 t_2 t_3^2 t_4 x^2 y^3 - 
 t_1^2 t_3 t_4 x^3 y^3 - t_1^2 t_2 t_3 t_4 x^3 y^3 - t_1 t_2 t_3^2 t_4 x^4 y^3 \\&- 
 t_1 t_2 t_3 t_4^2 x^3 y^4 - t_1^2 t_2 t_3^2 x^4 y^4 - 
  t_1^2 t_2 t_4^2 x^4 y^4 + 2 t_1^2 t_2 t_3^2 t_4^2 x^4 y^4 - 
 t_2^2 t_3^2 t_4^2 x^4 y^4 \\&+ t_1^3 t_2 t_3 t_4^2 x^5 y^4 + 
 t_1 t_2^2 t_3^3 t_4^2 x^5 y^4 + t_1^3 t_2 t_3^2 t_4 x^4 y^5 + 
 t_1 t_2^2 t_3^2 t_4^3 x^4 y^5 + t_1^2 t_2 t_3^3 t_4 x^5 y^5\\& + 
 t_1^2 t_2^2 t_3^3 t_4 x^5 y^5 + t_1^2 t_2 t_3 t_4^3 x^5 y^5 + 
 t_1^2 t_2^2 t_3 t_4^3 x^5 y^5 + t_1^3 t_2 t_3^2 t_4 x^6 y^5 + 
 t_1 t_2^2 t_3^2 t_4^3 x^6 y^5 \\&+ t_1^3 t_2 t_3 t_4^2 x^5 y^6 + 
 t_1 t_2^2 t_3^3 t_4^2 x^5 y^6 - t_1^4 t_2 t_3^2 t_4^2 x^6 y^6 + 
 2 t_1^2 t_2^2 t_3^2 t_4^2 x^6 y^6 - t_1^2 t_2^2 t_3^4 t_4^2 x^6 y^6\\&
  - 
 t_1^2 t_2^2 t_3^2 t_4^4 x^6 y^6 - t_1^3 t_2^2 t_3^3 t_4^2 x^7 y^6 - 
 t_1^3 t_2^2 t_3^2 t_4^3 x^6 y^7 - t_1^2 t_2^2 t_3^3 t_4^3 x^7 y^7\\& - 
 t_1^2 t_2^3 t_3^3 t_4^3 x^7 y^7 - t_1^3 t_2^2 t_3^2 t_4^3 x^8 y^7 - 
 t_1^3 t_2^2 t_3^3 t_4^2 x^7 y^8 + t_1^4 t_2^3 t_3^4 t_4^4 x^{10} y^{10},
\end{aligned}
$$
$$\begin{aligned}
\delta(x,y,t)
=&
(1-t_1 t_3 x)  (1-x^2 ) (1-t_2 x^2 ) (1-t_1 t_4 y ) (1-y^2 ) (1-t_2 y^2 )(1-x^2 y^2 )
 (1-t_3 t_4 x y ) \\& (1-t_2 t_3 t_4 x y )
 (1-t_1 t_4 x^2 y ) (1-t_1 t_3 x y^2 ) (1-t_1^2 x^2 y^2 )(1-t_2 t_3^2 x^2 y^2 ) (1-t_2 t_4^2 x^2 y^2)
 \end{aligned}
$$
\end{pro}
\begin{proof}
Let $p$ and $q$ be the polynomials such that 
$$\sum_{i,j,k=0}^\infty 
\mu_{i,j,k,n_1,n_2,m}(n_1, n_2; m) t_1^{n_1} t_2^{n_2} t_3^m 
x^i y^j z^k =\frac{p(x,y,z,t)}{q(x,y,z,t)}.
$$
They may be computed explicitly using proposition \ref{pro: mu ijk n1 n2 m}.
Set $t'=(t_1,t_2,t_3),$
and $$f(x,y,t',t_4)=\sum_{i,j=0}^\infty\sum_{n \in \Z_{\ge 0}^4} \mu_{i,j}(n) t_1^{n_1}t_2^{n_2}t_3^{n_3}t_4^{n_4}x^iy^j.$$ By
regarding $V_{(1,0;0;1)}$ as two copies of $V_{(1,0;0)}$ with the standard
torus of the second $SL_2$ acting with eigenvalues $\tau$ and $\tau^{-1},$ 
say, we see that
 $$
 \frac{\tau f(x,y,t',\tau) -\tau^{-1} f(x,y,t',\tau^{-1})
}{(\tau - \tau^{-1})}=
\sum_{r=0}^{n_4} \tau^{n_4-2r}=\frac{p(x,y\tau, y\tau^{-1}, t')}{q(x,y\tau, y\tau^{-1}, t')}.
$$
So it suffices to verify that 
$$\begin{aligned}
p(x,y\tau,& y\tau^{-1}, t')(\tau - \tau^{-1})\delta(x,y,t',\tau)
\delta(x,y,t',\tau^{-1})
\\&=q(x,y\tau, y\tau^{-1}, t')[\tau \nu(x,y,t',\tau)\delta(x,y,t',\tau^{-1}) -\tau^{-1} \nu(x,y,t',\tau^{-1})\delta(x,y,t',\tau)],\end{aligned}$$
which is easily done with a computer algebra system.
\end{proof}

\section{Local zeta integrals I}
\label{sec: local zeta integrals}
\subsection{Definitions and notation}
\label{ss: local zeta, defs and notation}
The next step in the analysis of our global integral 
is to study the corresponding local zeta integrals. 
We introduce a ``local'' notation which will be used throughout 
sections \ref{sec: local zeta integrals} and  \ref{section: unramified calculation}, \ref{section: local zeta integrals ii}
In these section $F$ is a local field which may be archimedean or
nonarchimedean. 
Abusing notation, we denote the $F$-points
of an $F$-algebraic group $H$ by $H$ as well.
We fix an additive character $\psi$ of $F,$ 
and define a character $\psi_N:N \to \C$ 
by the same formula used in the global setting.
Similarly, if we fix 
a triple $\chi=(\chi_1, \chi_2, \chi_3)$ of 
characters of $F^\times,$ and 
$s\in \C^3,$ 
then formula \eqref{ def: (chi; s)} now defines a 
character of $M_P.$ 
We write $\Ind_{P}^G(\chi; s)$ for the corresponding 
(unnormalized) induced representation ($K$-finite vectors, 
relative to some fixed maximal compact $K$).
We shall assume that the characters
in $\chi$ are unitary, but not that they are normalized,
and define $(\chi;s)$ for $s \in \C^2$ by the convention
 $s_3 = \frac{3s_1+3s_2}2.$ 
 Thus we have a two parameter family of 
 induced representations and let $\Flat(\chi)$ 
 denote the space of flat sections. 
 
Let $\c S( \Mat_{4 \times 2})$ 
be the Bruhat-Schwartz space, which we equip with 
an action $\omega_{\psi,l}$ 
of $R_Q$ as in the global setting, 
and define $I_0 : \c S(\Mat_{4\times 2}) \times R_Q \to \C$
by replacing $\A$ by $F$ in \eqref{I_0(phi, g) def}.

Next, take $\pi$ to be a $\psi_N$-generic 
irreducible 
admissible representation of $C_Q$ 
with $\psi_N$-Whittaker model
$\c W_{\psi_N}(\pi),$ and with central character
$\chi_1^{-3}\chi_2^{-3} \chi_3^2.$ 
 
For $W \in \c W_{\psi_N}(\pi), f \in \Flat(\chi)$
and $\phi \in \c S(\Mat_{4\times 2}),$ define
the corresponding local zeta integral 
to be the local analogue of 
\eqref{eq: global unfolded}, namely:
\begin{equation}
\label{eq: local zeta integral def}
I(W, f, \phi; s) := \int_{ZU_4\bs C_Q} W(g) \int_{U^w_Q \bs U_Q}
f_{\chi;s}(wug)I_0(\phi, ug)\, du \, dg.
\end{equation}
In addition to the above notation, for $1 \le i,j \le r, i \ne j,$ let  
$x_{ij}(r)=I_{12}+rE_{i,j} - rE_{13-j,13-i},$
where $I_{12}$ is the $12\times 12$ identity matrix and 
$E_{i,j}$ is the matrix with a one at the $i,j$ entry and 
zeros everywhere else, and let 
$$
\Xi_0 := \bpm 0&0 \\ 1& 0 \\ 0 & 1 \\ 0 & 0 \epm, \qquad
\Xi_2(a):=\bpm a_1&a_2 \\ 0&a_4\\0&0\\0&0\epm, \qquad
(a = (a_1,a_2,a_4) \in F^3).
$$

\subsection{Inital computations}
\label{ss: initial computations}
In this section we carry out some initial 
computations with local zeta integrals which will 
be used in both the proof of convergence in section \ref{ss: convergence} and in the unramified computations 
carried out in section \ref{section: unramified calculation}.

The image of the function $x_{23}$ maps 
isomorphically onto the one dimensional quotient 
of $U_4\bs N,$ 
and the function $g \mapsto f_s^\circ(wg)$ is invariant by the image
of $x_{23}$ on the left.
Moreover, $W(x_{23}(x_4)g)
= \ol\psi(x_4) W(g),
$
while
$$\left[\omega_\psi\left( x_{23}(x_4) ug \right).\phi\right]
(\Xi_0 + \Xi_2(x_1,x_2,0)) 
=\left[\omega_\psi\left( ut \right).\phi\right]
(\Xi_0 + \Xi_2(x_1,x_2,x_4))$$
Let 
$$\begin{aligned}
\III(\phi) &:= 
\int_{F^3} \phi
\bpm r_1&r_2 \\ 1&r_4\\0&1\\0&0\epm\ol\psi(r_1+r_4) \, dr,
\qquad 
( \phi \in \c S(\Mat_{4\times 2}))
\\
\II(f,\phi, s) 
&:= 
\int_{U_Q^w\bs U_Q}
f_{\chi;s}(wu) \III(\omega_{\psi, l}(u).\phi) \, du,\qquad 
( \phi \in \c S(\Mat_{4\times 2}), f \in \Flat(\chi))
\\
I_1( W, f, \phi; s) 
&:= \int_{Z\bs T} 
W(t)   
\II( R(t). f, \omega_{\psi, l}(t). \phi, s) 
\ \delta_B^{-1}(t)\, dt,\end{aligned}$$
where $\phi \in \c S(\Mat_{4\times 2}), f \in \Flat(\chi), 
W \in \c W_{\psi_N}(\pi),$ and $R$ is right translation.
Then expressing Haar measure on $C_Q$ in 
terms of Haar measures on $T, N$ and $K,$
and using $x_{23}$ to parametrize $U_4\bs N$ 
yields 
\begin{equation}\label{eq: I = int of I1 over K}
I(W,f, \phi; s) 
= \int_K I_1(R(k).W, R(k).f, \omega_{\psi, l}(k).\phi; s) \, dk
\end{equation}
where  $K$ is the maximal 
compact. 

The integral $\III(\phi)$ is absolutely 
convergent. Indeed, $\III(\phi) = \phi_1(\Upsilon_0),$
where $\phi_1$ is the Schwartz function obtained 
by taking Fourier transform of $\phi$ in three of the 
eight variables, and $\Upsilon_0$ is a matrix 
with entries $0$ and $1.$ We study the dependence 
on $u \in U_Q^w \bs U_Q$ and $t \in Z \bs T$ using 
the local analogues of 
\eqref{eq: Weil rep fmlas}. A remark is in order, regarding the Weil index $\gm_{\psi, \det m_Q^1( g_1, g_2)}$ which appears in the third formula.
In order to reconcile the local and global cases, one should think of this 
as the ratio $\gm_{\psi, \det m_Q^1( g_1, g_2)}/\gm_{\psi, 1}.$ The denominator can be omitted because the global $ \gm_{\psi, 1}$ is trivial. 
In the local setting $\gm_{\psi, 1}$ may not be trivial, but 
$\gm_{\psi, a^2} = \gm_{\psi, 1}$ for any $a,$ and $\det m_Q^1( g_1, g_2)$
is always a square, so the ratio is always trivial.

Now, let $U_0 \subset U_Q$ be the subgroup 
corresponding to the variables, 
$x_1,x_2,x_4, y_3, y_5, y_6,$ $z_1,z_2,z_3$ and 
$z_5.$  That is, the subset in which 
all other variables equal zero. 
Let $U_7 \subset U_Q$ be the subgroup 
defined by the condition that each 
variable listed above is $0,$ and, in addition, 
$y_7=y_8=0.$ 
Then $U_0U_7$ maps isomorphically onto the quotient
$U_Q^w\bs U_Q.$
We parametrize $U_0$ and $U_7$ 
using the coordinates inherited from $U_Q.$ 
A direct computation 
using \eqref{eq: Weil rep fmlas}  shows that for
$u_0 \in U_0, u_7 \in U_7,$
and $t \in T,$ 
$\III(\omega_\psi( tu_0  u_7).\phi)$
is equal to  
\begin{equation}\label{eq: III(u0 t u7 phi)}
 |t ^{\beta_1 + \beta_2 + \beta _4 }|
|\det t|^{\frac 12}
\psi(x_1t^{\beta_1} + x_4t^{\beta_4} - y_3 t^{-\beta_3}+y_6t^{-\beta_6}+z_1+z_5)
\phi' 
\bpm y_1+t^{\beta_1} & y_2 
\\ x_3+t^{-\beta_3}
&y_4+t^{\beta_4} \\ x_5 & x_6+t^{-\beta_6}\\ 
x_7  & x_8  \epm,
\end{equation}
where  $\phi'$ is the Schwartz function
obtained by taking the Fourier transform
of $\phi$ (relative to $\psi$) in $x_1, x_2$ and $x_4.$ 
If we define
$\psi_{U_0,t}(u_0):=\psi(x_1t^{\beta_1} + x_4t^{\beta_4} - y_3 t^{-\beta_3}+y_6t^{-\beta_6}+z_1+z_5),$
and we  define 
$\delta(t) \in U_7$ 
and $\pi_7: U_7 \to \Mat_{4\times 2}$ by 
$$
\delta(t)= 
u_Q\left( \bpm0& 0 
\\ t^{-\beta_3}
&0 \\ 0 &t^{-\beta_6}\\ 
0  & 0  \epm
,\bpm  0& 0\\ 0
&0 \\ t^{\beta_4} &   0
\\ 
0  &   t^{\beta_1}  \epm,0
\right)\in U_7, \qquad 
\pi_7(u_7)
=
\bpm y_1 & y_2 
\\ x_3
&y_4 \\ x_5 & x_6\\ 
x_7  & x_8  \epm,
$$
then
\eqref{eq: III(u0 t u7 phi)}
becomes 
$|t ^{\beta_1 + \beta_2 + \beta _4 }|
|\det t|^{\frac 12}\psi_{U_0,t}(u_0) \phi'(\pi_7(\delta(t)^{-1}u_7)).$

The projection $\pi_7$ has a two dimensional kernel 
corresponding to the 
variables $z_4$ and $z_6.$ 
Let $U_8$ denote this kernel 
and choose a subset $U_8'$ of $U_7$ which 
contains $\delta(t)$ and maps  
isomorphically onto the quotient. Then we can 
parametrize $\II( R(t). f, \omega_{\psi, l}(t). \phi, s)$
as a triple integral over $U_0 \times U_8 \times U_8'.$
The $U_8'$ integral is convergent because $\phi'$ is 
Schwartz, after a change of variables it becomes
$$
\phi'*_1f_{\chi;s}(wu_0u_8 \delta(t) ):=
\int_{U_8'} 
f_{\chi;s}(wu_0u_8 \delta(t) u_8') \phi'(\pi_7(u_8')) \, du_8.
$$
Thus, conjugating $t$ from right to left, 
and making a change of variables yields
$I_1(W, f, \phi; s)
= I_2( W, \phi*_1f;s),$
where $I_2( W, f,\phi;s )$ is defined as 
$$ 
\int_{Z\bs T} 
W(t) 
\delta_B^{-1}(t) |\det t|^{\frac 12}
| t^{\beta_1+\beta_2+\beta_4}|
\Jac_1(t) (\chi;s)(wtw^{-1}) 
f_{\chi;s}(wu_0u_8 \delta(t) )\, dt,
$$
with $\Jac_1(t)$ being the ``Jacobian''
of the change of variables $u_0 \to tu_0t^{-1}, \ 
u_7 \to tu_7 t^{-1}.$
Notice that $\phi*_1f$
 is simply another smooth section of the 
same family of induced representations,
and that if $f_{\chi;s}$ and $\phi$ are both 
 unramified, then $\phi' = \phi$ and 
 $\phi'*_1f_{\chi;s}= f_{\chi;s}.$ Thus, we may dispense with 
 the integral over $u_8'.$
 
 Next, we dispense with the integral over $u_8.$ To 
 do this, we use \cite{Dixmier-Malliavin} to replace
 $f_{\chi;s}$ by a sum of sections of the form 
 $$
 \phi_2*_2 f'_{\chi;s}(g) := \int_{F^2} f_{\chi;s}(g x_{24}(y_1)x_{34}(y_2)) \phi_1(y_1,y_2) \, dy,\qquad 
 (s\in \C^2, g\in G).
 $$
 Now, let
 $$[\II_2.f_{\chi;s}](g) :=
 \int_{U_0}f_{\chi;s}(w u_0g)\psi_{U_0,t}(u_0) \, du_0, 
 \qquad u_9(y_1,y_2) := x_{24}(y_1)x_{34}(y_2).$$
  conjugating $u_9(y_1,y_2)$ from 
  right to left shows that 
 $$[\II_2.f_{\chi;s}]( \delta(t) u_8 u_9(y_1,y_2))
 =
 \psi(-y_1z_6 - y_2(z_4-t^{\alpha_1}))
 \cdot [\II_2.f_{\chi;s}](  \delta(t) u_8).$$
 (Recall that $z_6$ and $z_4$ are coordinates on $U_8.$)
 But then 
$$\int_{F^2}[\II_2.\phi_2*_2f_{\chi;s}]( \delta(t) u_8) \, dz 
= \int_{F^2}[\II_2.f_{\chi;s}]( \delta(t) u_8) 
\widehat\phi_2(z_6, z_4-t^{\alpha_1})
\, dz,$$
which we may write as
$[\II_2.\widehat\phi_2*_3f_{\chi;s}](  \widetilde\delta(t)),$
where $*_3$ is the action of $\c S(U_8)$ by convolution, 
and $\widetilde\delta(t)= \delta(t) x_{29}(t^{\alpha_1}).$ 
Notice that $\widehat\phi_2*_3f_{\chi;s}$ is again another 
smooth 
section of the same family of induced representations.
Note also that if $f$ is spherical then taking $\phi_2$
to be the characteristic function of $\f o^2$ 
gives $\widehat\phi_2*_3f_{\chi;s}= f_{\chi;s}.$ 

Thus, we are reduced to the study of the integral
\begin{equation}
\label{eq: def of I3}I_3( W, f; s):=
\int_{Z\bs T} W(t)
\delta_B^{-1/2}(t) \nu_s(t)
 \II_2.f_{\chi; s}( \widetilde\delta(t))\, dt,
  \end{equation}
where $\nu_s(t):=
 \delta_B^{-1/2}(t) |\det t|^{\frac 12}
| t^{\beta_1+\beta_2+\beta_4}|
\Jac_1(t) (\chi;s)(wtw^{-1}).$
\label{section with w1 w2 and s3}
Write $w=
w_1w_2w_3,$ where $w_1=w[634],$ $w_2=w[3236514]$
and $w_3=w[2356243564].$
Write $U$ for the unipotent radical of our standard 
Borel of $G,$ and $U^-$ for the unipotent radical of the 
opposite Borel. For $w \in W,$ let $U_w = U \cap w^{-1} U^- w.$ 
Then $w_3 U_0 w_3^{-1} = U_{w_1w_2}= w_2^{-1}U_{w_1} w_2 U_{w_2}.$  
For ${\bf c}:=(c_1, \dots, c_6) \in F^6,$
define a character $\psi_{{\bf c},0}$ of $U_0$ in terms
of the standard coordinates on $U_0$ by 
$
\psi_{{\bf c},0}(u_0):=
\psi(c_1x_1+c_2x_4+c_3y_6-c_4y_3+c_5z_1+c_6z_5).
$  
Notice that $\psi_{U_0,t}$ is obtained by taking 
${\bf c}=(t^{\beta_1},t^{\beta_4}, t^{-\beta_3}, t^{-\beta_6},1,1).$
In terms of the entries $u_{ij}$ we have
$
\psi_{{\bf c},0}(u) = \psi(c_1u_{15} +c_2 u_{26}+c_3 u_{27}
-c_4 u_{38} + c_5 u_{19} + c_6 u_{2,10}),
$
Now, 
$w_3$ 
corresponds to the permutation 
$(2,4,11,9)(3,8,10,5).$
It follows that 
$u_0'\mapsto \psi_{{\bf c},0}( w_3^{-1} u_0' w_3)$ is the 
character
of $w_3U_0w_3^{-1}$
given by $u\mapsto \psi( c_1u_{13} + c_2 u_{46}+c_3 u_{47}+c_4 u_{35} + c_5 u_{12} + c_6 u_{45}).$
In particular, its restriction to 
$w_2^{-1}U_{w_1} w_2$ 
is trivial.  Let $\psi_{{\bf c},2}$ denote the 
restriction to $U_{w_2}.$  
Then 
\begin{equation}\label{int U0 = int U1 int U2}
\int_{U_0} f_{\chi;s}(wug) \psi_{{\bf c},0}(u) \, du 
= \int_{U_{w_2}} \int_{U_{w_1}} f_{\chi;s}(w_1u_1w_2 u_2w_3g)  \, du_1\,\psi_{{\bf c},2}(u_2)\, du_2,\end{equation}
and the $u_1$ integral is a standard intertwining operator, 
$M(w_1^{-1}, \chi; s):\Ind_P^G(\chi;s)\to \nInd_{B_G}^G ((\chi;s)\delta_{B_G}^{-1/2})^{w_1},$
where  $((\chi;s)\delta_{B}^{-\frac12})^{w_1}(t):= ((\chi;s)\delta_{B}^{-\frac12})(w_1 t w_1^{-1}),$
and $\nInd$ denotes normalized induction.

Now let $w_4=w[32365]$ so that $w_2=w_4w[14].$ 
Also, let $w_3' = w[14]w_3.$
Observe that $w_4$ is the long element of the Weyl 
group of a standard Levi subgroup  of  $GSO_{12}$
which is isomorphic 
to $GL_1 \times GL_3 \times GSO_4.$ 
For $c_1, \dots, c_4 \in F,$ define 
a character $\psi_{{\bf c},4}$ of $U_{w_4}$ by 
$\psi_{{\bf c},4}
(u) = \psi(
c_1u_{23}+c_4u_{34}+c_3u_{56}+c_2u_{57}),$
and for $f_{\chi;s;w_1}\in \nInd_{B_G}^G ((\chi;s)\delta_{B_G}^{-1/2})^{w_1},$
let 
$$
\c J_{\psi_{{\bf c},4}}.f_{\chi;s;w_1} ( g)
= \int_{U_{w_4}} f_{\chi;s,w_1}^\circ(w_4 ug) \psi_{{\bf c},4}(u)\, du,
$$
which is a Jacquet integral for this Levi subgroup. Then \eqref{int U0 = int U1 int U2} equals
$$ \int_{F^2} [\c J_{\psi_{{\bf c},4}}\circ M(w_1^{-1}, \chi;s).f_{\chi;s}](x_{21}(r_1)x_{54}(r_2)w_3'g)\psi(c_5r_1+c_6r_2) \, dr.
$$

\section{Unramified Calculation}\label{section: unramified calculation}
We keep the notation from section \ref{ss: local zeta, defs and notation}, 
and assume further that 
 $F$ is nonarchimedean, 
with ring of integers $\f o$  having unique maximal ideal 
$\f p.$ 
We fix a generator $\f w$ for $\f p.$ 
 The absolute value 
on $F$ is denoted $|\ |$ and
 normalized so that $|\f w|=q:=\# \f o/\f p.$
 The corresponding $\f p$-adic valuation is 
 denoted $v.$
Moreover, we assume that $K=G(\f o)$, and 
that the representation
$\pi$ and characters $\chi_i, \ i=1,2,3$ are 
 unramified, and we let $W_\pi^\circ, f^\circ$ and $\phi^\circ$
 denote the normalized spherical elements
 of 
$\c W_{\psi_N}(\pi), \Flat(\chi)$ and $\c S(\Mat_{4\times 2}),$
respectively.

The (finite Galois form of the) $L$-group of
$GSp_4 \times GSO_4$ is $GSpin_5(\C) \times GSpin_4(\C).$
Indeed, one may define $GSpin_{2n+1}$ (resp. $GSpin_{2n}$) as the 
reductive group with root datum dual to that of $GSp_{2n}$ (resp. $GSO_{2n}$). However, both $GSpin$ groups appearing here can be understood more
explicitly via ``coincidences of low rank.'' 
Indeed, a simple change of $\Z$-basis reveals that the root 
datum of $GSp_4$ is in fact {\it self } dual. Thus $GSpin_5$ is just 
$GSp_4$ in another guise. Note, however, that the isomorphism
of $GSp_4$ with its own dual group does not respect the standard 
numbering of the simple roots.

Next, we can realize $GSO_4$ (resp. $GSpin_4$) as a quotient 
(resp. subgroup) of $GL_2 \times GL_2.$ 
Indeed, we can realize $GSO_4$ as the 
similitude group of 
 the four dimensional quadratic
space $(\Mat_{2\times 2}, \det).$ 
Letting $GL_2\times GL_2$ act by 
$(g_1, g_2) \cdot X = g_1 X  {}^t \!g_2$
induces a surjection $GL_2 \times GL_2 \to 
GSO_4$ with kernel 
$\{ (aI_2, a^{-1}I_2): a \in GL_1\},$ and thence
a bijection between representations of 
$GSO_4$ and pairs of representations
of $GL_2$ with the same central character.  
By duality, this  
  induces an isomorphism of 
 $GSpin_4$ with 
$\{ (g_1, g_2) \in GL_2\times GL_2: \det g_1 = \det g_2\}.$  
We remark that the induced map
$GSpin_4 \to SO_4$ is {\it not} the restriction of 
 our chosen 
map  $GL_2 \times GL_2 \to GSO_4.$  

We regard $GSp_4 \times GSO_4$ as a subgroup of $M_Q$ 
containing $C_Q$ in the obvious way, and  regard its $L$ group 
as a subgroup of $GSp_4 \times GL_2\times GL_2.$ 
We make the identification in such a way that the first $GL_2$ 
corresponds to the fifth simple root of $G=GSO_{12}$
and the second $GL_2$ corresponds to the sixth simple root of $G.$

Let $St_{GSp_4}$ denote the standard representation of $GSp_4.$
It may also be regarded as the spin representation
of $GSpin_5.$ For this reason, the associated $L$ function 
is often called the ``Spinor $L$ function.''
We regard $St_{GSp_4}$ as a representation of of the $L$ group via  
projection   onto the $GSp_4(\C)$ factor, and 
let $St_{GSp_4}^\vee$ denote the 
dual representation. 
Let $St_{GL_2^{(1)}}$ (resp. $St_{GL_2^{(2)}}$)
denote the representations of the $L$ group 
obtained by composing the standard representation of $GL_2$
with projection onto the 
first (resp. second) $GL_2(\C)$ factor.
\begin{thm}\label{thm:  main local}
Let 
 \begin{equation}\label{eq:N(s, chi)}
 \begin{aligned}
N(s, \chi)=&
L(s_1-s_2, \frac{\chi_1}{\chi_2})
L(s_1-s_2-1, \frac{\chi_1}{\chi_2})
L(s_1-s_2-2, \frac{\chi_1}{\chi_2})
L(s_1+s_2-2, \chi_1\chi_2)
\\ &\times 
L(s_1+s_2-3, \chi_1\chi_2)
L(s_1+s_2-4, \chi_1\chi_2)
L(2s_2, \chi_2^2)
L(2s_1-6, \chi_1^2).\end{aligned}\end{equation}
(Local $L$ functions.  The corresponding product of global zeta functions is 
the normalizing factor of the Eisenstein series.)
Then $I(W_\pi^\circ, f^\circ, \phi^\circ; s)$ equals
$$ 
\frac{L(\frac{s_1-s_2}2-1, \pi, St_{GSp_4}^\vee
\otimes St_{GL_2^{(1)}} \times \frac{\chi_3}{\chi_1 \chi_2^2}) 
L(\frac{s_1+s_2}2-2, \pi, St_{GSp_4}^\vee
\otimes St_{GL_2^{(2)}} \times \frac{\chi_3}{\chi_1 \chi_2})}{N(s,\chi)}
$$
\end{thm}

\begin{proof}[Reduction of the general case
to the special case of trivial  
characters]
For purposes of this 
proof,  write $\lambda_H$ for the similitude 
rational character of $H$ where
$H = GSO_{12}, GSp_4,$ or $GSO_4.$ 
Our embedding 
$(GSp_4 \times GSO_4)^\circ \hookrightarrow 
GSO_{12}$  is such that 
$
\lambda_{GSO_{12}}( g, h) = \lambda_{GSp_4}(g)^{-1}= \lambda_{GSO_4}(h), 
$
and the projection 
$p: GL_2 \times GL_2 \to GSO_4$ is such that 
$\lambda_{GSO_4}(p(g_1, g_2)) = \det g_1 \det g_2.$

Write $\pi = \Pi \otimes \tau_1\otimes \tau_2$ 
where $\Pi$ is an unramified representation 
of $GSp_4$ and $\tau_1$ and $\tau_2$ are
unramified representations of $GL_2$ with 
the same central character (so that $\tau_1 \otimes 
\tau_2$ is a representation of $GSO_4$). 
Write 
$\tau_i = \tau_{i,0}\otimes |\det |^{t_1}$ where 
 $\tau_{i,0}$ is an unramified representation 
 of $GL_2$ with trivial central character 
 for $i=1,2$, 
 and $t_1$ is a complex number, 
 and write 
$\Pi = \Pi_0 \otimes |\lambda_{GSp_4}|^{t_2},$
where $\Pi_0$ is an unramified representation of 
$GSp_4$ with trivial central character
and $t_2$ is a complex number. 
Then, as representations of $C_Q=(GSp_4 \times GSO_4)^\circ,$ 
$$
\pi =\pi_0\otimes |\lambda_{GSO_{12}}|^{t_1-t_2}, 
\qquad \text{ where }
\pi_0= \Pi_0 \otimes \tau_{1,0}\otimes\tau_{2,0}.
$$
The operation of twisting 
$\Pi_0$ by $|\lambda_{GSp_4}|^{t_2}$ 
to obtain $\Pi$ corresponds, on the 
$L$ group side, to multiplying the corresponding
semisimple conjugacy class in $GSp_4(\C)$
by $q^{-t_3}I_4.$ Likewise, 
the operation of twisting 
$\tau_{i,0}$ by $|\det |^{t_1}$ 
corresponds to multiplying by $q^{-t_1}$ in $GL_2(\C)$ for $i=1,2.$  
If $\eta$ is the unramified character $\eta(a) = |a|^r$, then 
$$L\left(u, \pi, St^\vee_{GSp_4} \otimes St_{GL_2^{(i)}}
\times 
 \eta
\right)
= L\left(u-t_2+t_1+r, \pi_0, St^\vee_{GSp_4} \otimes St_{GL_2^{(i)}}\right).$$

For $i=1,2,3$ the unramified character
$\chi_i$ is given by $|\ |^{r_i}$ 
for some $r_i \in \C.$ 
If $s = (s_1, s_2, \frac{3(s_1+s_2)}2) \in \C^3,$ then 
let $s' = (s_1+r_1, s_2+r_2,  \frac{3(s_1+r_1+s_2+r_2)}2),$
and let $\chi_0$ be the triple consisting of three copies of the 
trivial character.
 Then it follows directly from the definitions that 
$f^\circ_{\chi; s} = f^\circ_{\chi_0; s'} \cdot |\lambda_{GSO_{12}}|^{r_3 - \frac{3r_1+3r_2}2}.$ The general case now follows from the case $t_1=t_2=r_1=r_2=r_3=0.$

Recall that $I(W, f, \phi; s)$ is only defined when the triple $\chi$ and the central character of $\pi$ are compatible.
In the present notation the compatibility condition 
 is that 
$-3r_1-3r_2 +2 r_3 + 2( t_1-t_2) = 0.$
But then $W_\pi \cdot f^\circ_{\chi; s} =W_{\pi_0}\cdot f^\circ_{\chi_0; s'}.$
The general case now reduces to the special case 
when $\chi = \chi_0$ and $\pi=\pi_0.$
\end{proof}
\begin{rem}
We may now  assume
 that 
 $\pi = \Pi \otimes \tau_1 \otimes \tau_2,$ 
 where 
 the central characters of $\Pi, \tau_1$ and
 $\tau_2$ are all trivial. Thus, $\pi$ 
 may be regarded 
 as an unramified representation of  $SO_5 \times PGL_2 \times PGL_2$  
and corresponds to a  semisimple 
conjugacy class in 
$Sp_4(\C) \times SL_2(\C) \times SL_2(\C).$
Define $St_{Sp_4}, St_{SL_2^{(1)}}$ and 
$St_{SL_2^{(2)}}$ as the restrictions of 
$St_{GSp_4}, St_{GL_2^{(1)}}$ and $St_{GL_2^{(2)}}$, respectively.  Note that all three are
self dual representations.  In particular, 
$L(u, \pi, St_{GSp_4}^\vee \otimes St_{GL_2^{(i)}})
= L(u, \pi, St_{Sp_4} \otimes St_{SL_2^{(i)}})$
for all $u \in \C.$
\end{rem}
\begin{proof}[Proof in the special case of trivial characters]
As $\chi$ is trivial we  write  $f^\circ_s$ instead 
of $f^\circ_{\chi; s}.$  Let $I(s, \pi):= I(W_\pi^\circ, f^\circ, \phi; s).$
As we've seen in section \ref{ss: initial computations},
it is equal to $I_3(W_\pi^\circ, f^\circ; s),$ defined by 
\eqref{eq: def of I3}.
 Let ${\bf I}_1(s; c_1, \dots, c_6)$ equal
$$ \int_{U_0}
f_{\chi;s}(wu_0
)
\psi_{{\bf c},0}(u_0) 
 \,du_0.
$$
Let $a=t^{\beta_1}, b=t^{\beta_4}, c=t^{-\beta_6},
d=t^{-\beta_3}.$
Note that  
$\beta_1 - \beta_3 = \alpha_1,$
$\beta_4-\beta_3 = \alpha_5,$
and $\beta_4-\beta_6= \alpha_2.$
Moreover, the characters $-\beta_3-\beta_6$ 
and $\alpha_6$ are not identical on 
$T_G,$ but they 
have the same restriction 
to the maximal torus $T$ of $C_Q.$
It follows that 
for $t$ in the support of $W_\pi^\circ,$ 
the quantities 
$|ad|, |bd|, |bc|$ and $|cd|$ are all $\le 1.$

By plugging in the Iwasawa decomposition
for $w_3\widetilde \delta(t),$ 
we find that 
\begin{equation} 
\label{pro:3.1} \II_2.f^\circ_s(\wt \delta(t))=
 \begin{cases}
{\bf I}_1(s;a,b,c,d,1,1), 
& |a|,|b|,|c|,|d| \le 1,\\
|d|^{-s_1-s_2+4}{\bf I}_1( s, ad, bd, cd, 1,1,0), & |d| > 1, |a|,|b|,|c|\le 1,\\
|c|^{-s_1-s_2+4}{\bf I}_1(s, a, bc, 1, cd, 1, 0), & |c|>1, |a|, |b|,|d| \le 1,\\
|a|^{-s_1+s_2+2}|c|^{-s_1-s_2+4}{\bf I}_1(s, 1, bc, 1, acd, 0,0), &|a|,|c|>1, |b|,|d| \le 1,\\
|a|^{-s_1+s_2+2}{\bf I}_1(s, 1, b,c,ad, 0,1), & |b|,|c|,|d|\le 1, |a|>1,\\
|a|^{-s_1+s_2+2}|b|^{-s_1+s_2+2}{\bf I}_1( s, 1,1,bc, abd, 0,0), & |c|,|d|\le 1, |a|,|b| >1,\\
|b|^{-s_1+s_2+2}{\bf I}_1(s, a,1,bc,bd,1,0), & |d|,|c|,|a| \le 1, |b| > 1.
\end{cases}
\end{equation}
Now let $f_{s,w_1}^\circ$ denote the normalized spherical vector 
in $\nInd_{B_G}^G((\chi_0;s)\delta_{B}^{-\frac12})^{w_1},$
and let 
 $$Z_1(s):=
 \frac{\zeta(2s_2-1)}{\zeta(2s_2)}
\frac{\zeta(s_1-s_2-1)}{\zeta(s_1-s_2)}
\frac{\zeta(s_1+s_2-3)}{\zeta(s_1+s_2-2)}$$
 Then, $M(w_1^{-1}, \chi; s).f_s^\circ 
= Z_1(s)
 f_{s,w_1}^\circ,$
 by  the Gindikin-Karpalevic formula,
 and hence 
 $$
{\bf I}_1(s; c_1, \dots, c_6) 
= Z_1(s)
\int_{F^2}
\c J_{\psi_{{\bf c},4}} f_{s,w_1}^\circ(  x_{21}(r_1) x_{54}(r_2) ) \psi(c_5r_1+c_6r_2)
\,dr.$$

We remark that 
 $((\chi_0;s)\delta_{B}^{-\frac12})^{w_1}$ maps 
$\diag( \lambda t_1, \dots, \lambda t_6, t_6^{-1}, \dots, t_1^{-1})\in T_G$
 to 
\begin{equation}\label{chi s prime}
|t_1|^{s_1-5}|t_2|^{s_1-4}|t_3|^{s_2-2}|t_4|^{-s_2}|t_5|^{s_1-3}
|t_6|^{1-s_2}|\lambda|^{s_3-2s_2-\frac{13}2}.
\end{equation}
 
\begin{lem}\label{lem: i1 in terms of j}
Assume that each of $c_5,c_6$ is either zero or a unit.
Assume further that 
if $c_6 =0$ then at least one of $c_2, c_3, c_4$ is a unit, 
and that if $c_5=0$ then $c_1$ is a unit, 
and set $J_{c_1,c_2, c_3, c_4}= \c J_{\psi_{{\bf c},4}} f_{s,w_1}^\circ( I_{12}).$
Then ${\bf I}_1(s; c_1, \dots, c_6)/Z_1(s)$ equals
$$
J_{c_1,c_2, c_3, c_4}-q^{-s_1-s_2+4}J_{ \frac{c_1}{\f w}, c_2, c_3, c_4}
- q^{-2s_1+6}J_{c_1,\frac{c_2}{\f w}, \frac{c_3}{\f w}, \frac{c_4}{\f w}}
+q^{-3s_1-s_2+10}
J_{\frac{c_1}{\f w},\frac{c_2}{\f w}, \frac{c_3}{\f w}, \frac{c_4}{\f w}}.
$$
\end{lem}
\begin{rem}
Observe that all the sextuples $c_1, \dots, c_6$ appearing 
in  \eqref{pro:3.1} satisfy the conditions of lemma 
\ref{lem: i1 in terms of j}.
\end{rem}
\begin{proof}
There exist cocharacters $h_i: GL_1 \to T_G,\; (i=1,2)$ such that 
$\la h_1, \alpha_i\ra = \delta_{i,1}$
and $\la h_2, \alpha_i \ra = \delta_{i,4}$ (Kronecker $\delta$). 
It follows that $
 \c J_{\psi_{{\bf c},4}} f_{s,w_1}^\circ(   x_{21}(r_1) x_{54}(r_2) )$
 depends only on $v(r_1)$ and $v(r_2).$ 
If $c_5$ is a unit, then 
$$
\int_{v(r_1) = -k}
\psi(c_5 r_1) \, dr_1 = \begin{cases}
-1, & k =1,\\ 0, & k >1,
\end{cases}
$$
and similarly with $r_2.$ 
Since both $f_{s,w_1}^\circ$ and $\psi$ are 
unramified,
it follows that ${\bf I}_1(s; c_1, \dots, c_6)$ equals
$$\c J_{\psi_{{\bf c},4}} f_{s,w_1}^\circ(I_{12})- \c J_{\psi_{{\bf c},4}} f_{s,w_1}^\circ(\f w^{-1}))-  \c J_{\psi_{{\bf c},4}} f_{s,w_1}^\circ(\f w^{-1}))+ \c J_{\psi_{{\bf c},4}} f_{s,w_1}^\circ( x_{21}(\f w^{-1})x_{54}(\f w^{-1})).$$
Plugging in the Iwasawa decomposition of $x_{21}(\f w^{-1})$
and/or $x_{54}(\f w^{-1})$ gives the result in this case. 

Now suppose that $c_5$ is not a unit.  Then it is zero and $c_1$
is a unit.  It follows that 
$$
\int_{F^2}
\c J_{\psi_{{\bf c},4}} f_{s,w_1}^\circ( x_{21}(r_1) x_{54}(r_2) ) \psi(c_5r_1+c_6r_2)
\,dr\\
=\int_{F}
\c J_{\psi_{{\bf c},4}} f_{s,w_1}^\circ( x_{54}(r_2) ) \psi(c_6r_2)
\,dr_2.
$$ 
Indeed the support of $J$ is contained in $UT_1K$ where
$K=GSO_{12}(\f o)$ is the maximal compact subgroup, 
 and 
$T_1$ is the set of torus elements $t$ with $|t^{\alpha_2}|\le 1.$
It follows easily from the 
Iwasawa decomposition that 
 $x_{21}(r_1) x_{54}(r_2)\in UT_1K$ if and only if $r_1 \in \f o.$
If $c_6$ is a unit then proceeding as before we obtain 
$$\int_{F}
\c J_{\psi_{{\bf c},4}} f_{s,w_1}^\circ(  x_{54}(r_2) ) \psi(c_6r_2)
\,dr_2=J_{c_1,c_2, c_3, c_4}
- q^{-2s_1+6}J_{c_1,\frac{c_2}{\f w}, \frac{c_3}{\f w}, \frac{c_3}{\f w}}.
$$
On the other hand, when $c_1$ is a unit then 
$J_{ \frac{c_1}{\f w}, c_2, c_3, c_4}
=
J_{\frac{c_1}{\f w},\frac{c_2}{\f w}, \frac{c_3}{\f w}, \frac{c_4}{\f w}}=0,$
and the stated result follows in this case as well.
Likewise, if $c_6$ is zero, then integration over
$r_2$ can be omitted and  $J_{c_1,\frac{c_2}{\f w}, \frac{c_3}{\f w}, \frac{c_4}{\f w}}=
J_{\frac{c_1}{\f w},\frac{c_2}{\f w}, \frac{c_3}{\f w}, \frac{c_4}{\f w}}=0,$  which gives the result in the 
remaining two cases.
\end{proof}

Now consider the subgroup of the torus 
consisting of all elements of the form 
\begin{equation}
\label{new torus coordinates}
t=\diag(
t_1t_2, t_2, 1, t_1^{-1}, 
t_3t_4, t_4, \frac{t_2}{t_4}, \frac{t_2}{t_3t_4}, 
t_1t_2, t_2, 1, t_1^{-1}).
\end{equation}
This subgroup maps isomorphically 
onto $Z\bs T.$ 
For elements of this torus and with coordinates as in 
\eqref{new torus coordinates}
we have
$$
\on{Jac}_1(t) 
=|t_1^2t_3t_4^2|^{-1}, 
\qquad 
\on{Jac}_2(t)
= \left|\frac{t_1^2 t_2^3}{t_3t_4^3}\right|\qquad 
\delta_B^{-\frac12}
=|t_1^2 t_2t_3t_4|^{-1} 
\qquad
|\det t|^{\frac 12} = |t_2^{-2} t_3^2 t_4^4|.
$$
$$
(\chi_0;s)( wtw^{-1}) 
= |t_1|^{s_1-s_2} |t_2|^{-s_1-3s_2+s_3} |t_3t_4^2|^{s_2}.
$$
Set $x = q^{-( \frac{s_1-3s_2}2)},$ $ y = q^{-s_2+1},$
and let $n_i$ be the $\f p=v(t_i)$
for $1 \le i \le 4.$  Then 

\begin{equation}
  \label{eq:11}
  \nu_s(t)
= x^{2n_1+n_2}y^{2n_1+n_3+2n_4}.
\end{equation}
For $l=(l_1, l_2, l_3, l_4)\in \Z^4,$ set
$$j_1(l) = 1 - x^{2 l_4+2}y^{ 2 l_4+2}-y^{2 l_1+2} - x^{2 l_1+2 l_4+4}y^{ 4 l_1+4 l_4+8} + x^{2 l_4+2}y^{2 l_1+4 l_4+6}+x^{2 l_1+2 l_4+4}y^{ 4 l_1+2 l_4+6}$$
$$
j_2(l ) = 1 - x^{2 l_2+2}y^{ 4 l_2+4}
\qquad
j_3(l) = 1 -  x^{2 l_3+2}y^{ 2 l_3+2}
$$
$$j(l ) = 
\begin{cases}j_1(l) j_2(l) j_3(l), & 
	l_i \ge 0 \forall i, \\  0, & \text{ otherwise.}\end{cases} 
$$
Then direct computation or the Casselman-Shalika formula shows that 
$$
\c J_{\psi_{{\bf c},4}}.f^\circ_{s,w_1}(I_{12})
= 
\frac{\zeta(s_1+s_2-4)^2}{\zeta(s_1+s_2-3)^2}
\frac{\zeta(s_1-s_2-2)^2}{\zeta(s_1-s_2-1)^2}
\frac{\zeta(2s_2-2)}{\zeta(2s_2-1)}j(l),
$$
where $l_i=v(c_i)$ for each $i.$
Hence, if 
$$i_1(l ) = j(l) -  x^{2}y^{ 4} j(l -(1,0,0,0)) -  x^{4}y^{ 6} j(l-(0,1,1,1)) + x^{6}y^{ 10} j(l-(1,1,1,1)),$$
then 
$${\bf I}_1(s; c_1, \dots, c_6) =\frac{\zeta(s_1-s_2-2)^2\zeta(s_1+s_2-4)^2\zeta(2s_2-2)i_1(n_1,n_2, n_3,n_4)}{\zeta(s_1-s_2)\zeta(s_1-s_2-1)\zeta(s_1+s_2-2)\zeta(s_1+s_2-3)\zeta(2s_2)},
$$
for all $c_1, \dots, c_6$ satisfying the conditions of lemma \ref{lem: i1 in terms of j}.
Consequently, $$
\II_2.f_{s}^\circ(\widetilde\delta(t))
= \frac{\zeta(s_1-s_2-2)^2\zeta(s_1+s_2-4)^2\zeta(2s_2-2)i(A,B,C,D)}{\zeta(s_1-s_2)\zeta(s_1-s_2-1)\zeta(s_1+s_2-2)\zeta(s_1+s_2-3)\zeta(2s_2)},
$$
where $A=v(t^{\beta_{1}}$), $B=v(t^{\beta_{4}})$, $C=v(t^{-\beta_{6}})$, and $D=v(t^{-\beta_{3}})$
and $i$ is defined piecewise in terms of $i_1$
according to the seven cases from  \eqref{pro:3.1}.
It is convenient to use an alternate parametrization. 
Let $ii(m_1, m_2, m_3, m_4)$ equal 
\begin{equation}\label{ii in terms of i}
i(m_1- \frac{-m_2+m_3+m_4}2, 
           \frac{m_2+m_3-m_4}2,
           \frac{m_2-m_3+m_4}2,
           \frac{-m_2+m_3+m_4}2)x^{2m_1+m_2}y^{2m_1+m_2+m_4}           \end{equation} if  $m_2+m_3+m_4$ is even and zero otherwise.
           Then $$\begin{aligned}
\nu_s(t) \II_2.f_s^\circ(\widetilde\delta(t))
&= \frac{\zeta(s_1-s_2-2)^2\zeta(s_1+s_2-4)^2\zeta(2s_2-2)ii(m)}{\zeta(s_1-s_2)\zeta(s_1-s_2-1)\zeta(s_1+s_2-2)\zeta(s_1+s_2-3)\zeta(2s_2)},
\end{aligned}
$$
where now $m_i=v(t^{\alpha_i})$
for $i=1,2,3,4.$            
Let $[m_1, m_2; m_3; m_4]$ or $[m]$ denote the 
trace of the irreducible representation of $^L(C_Q/Z):=Sp_4\times SL_2\times SL_2$
on which $Sp_4$ acts with highest weight $m_1\varpi_1+ m_2 \varpi_2,$ the first $SL_2$ acts with highest weight $m_3,$
and the second $SL_2$ acts with highest weight $m_4.$
Then 
$$
\delta_B^{-\frac12}(t)W_\pi(t) = [m_2, m_1; m_3; m_4](\tau_\pi),
$$
where $\tau_\pi$ is the semisimple conjugacy class 
in $^L(C_Q/Z)$ 
attached to $\pi.$
Note the reversal of order between $1$ and $2.$ 
The reason for this is that when $GSp_4$
is identified with its own dual group, 
the standard numberings for the two dual $GSp_4$'s are 
opposite to one another.  For example the coroot attached to 
the short simple root $\alpha_1$ is the long simple coroot, 
which makes it the long simple root of the dual group.

Now, let $Z_2(x,y)= (1-y^2)(1-x^2y^2)^2(1-x^2y^4)^2$. Then $j(n)$ is divisible by $Z$ for any $n.$
Also, for $\varrho$ the character of a finite dimensional 
representations of $^L(C_Q/Z),$ let  
$
L(u, \varrho ) = \sum_{i=0}^\infty 
u^k \Tr \sym^k(\varrho),$ 
Then theorem \ref{thm:  main local} is reduced to the following 
identity of power series over representation ring of $^L(C_Q/Z):$ 
\begin{equation}\label{id of power series over rep ring}
\frac 1{Z_2(x,y)} \sum_{m \in \Z_{\ge 0}^4}ii(m_2, m_1, m_3, m_4) [m]=Z_3(x,y) L(xy, [1,0;1;0])L(xy^2,[1,0;0;1]),
\end{equation}
where 
$Z_3(x,y)=(1-x^2y^2)(1-x^2y^4)(1-x^4y^6).$

Define polynomials $P_{m}(u,v)$ by 
$$
L(u, [1,0;1;0])L(v,[1,0;0;1])
=\sum_{m \in \Z_{\ge 0}^4}
P_m(u,v) [m].
$$
Then 
\eqref{id of power series over rep ring}
is equivalent to the family of  identities of polynomials, 
$$ (1-y^2)(1-x^2y^2)^3(1-x^2y^4)^3(1-x^4y^6)P_m(xy,xy^2) = ii(m_2, m_1, m_3, m_4) \; (\forall m \in \Z_{\ge 0}^4),$$
or to the identity of power series over a polynomial ring: 
\begin{equation}\label{power series identity}
\begin{aligned} \sum_{m \in \Z_{\ge 0}^4} ii(m_2, m_1, m_3, m_4)
 t_1^{m_1} t_2^{m_2} t_3^{m_3} t_4^{m_4}
&= Z_4(x,y)
\sum_{m \in \Z_{\ge 0}^4}
P_m(xy, xy^2) t_1^{m_1} t_2^{m_2} t_3^{m_3} t_4^{m_4}\\&=
 Z_4(x,y)
\frac{\nu(x,y,t)}{\delta(x,y,t)},
\end{aligned}
\end{equation}
where $\nu$ and $\delta$ are defined as in proposition 
\ref{pro: decomp of V times V'}, and $Z_4=Z_2Z_3.$ 

The identity \eqref{power series identity} can be proved as follows. 
Let $X = (X_1, X_2,X_3,X_4)$ and $Y= (Y_1, Y_2, Y_3, Y_4)$
be quadruples of indeterminates. Define polynomials, 
$$
\b j_1(x,y,X,Y) := 1 - x^2 y^2 X_4^2 Y_4^2 - y^2 Y_1^2 - x^2y^8  X_1^2 X_4^2 
Y_1^4
Y_4^4 + x^2 y^6   X_4^2 Y_1^2Y_4^4 + 
x^4 y^6   X_1^2X_4^2 Y_1^4Y_4^2
$$ 
$$
\b j_2(x,y,X,Y)= (1 - x^2 y^4 X_2^2 Y_2^4);
\qquad
\b j_3(x,y,X,Y) = (1 - x^2 y^2 X_3^2 Y_3^2);
\qquad \b j = \b j_1\b j_2\b j_3,
$$
so that 
for $k= (k_1, \dots, k_4) \in \Z_{\ge 0}^4,$ the polynomial
$j(k)$ is equal to  
$\b j(x,y,x^k,y^k),$ where  $x^k:=(x^{k_1}, \dots, x^{k_4})$
and $y^k:= (y^{k_1}, \dots, y^{k_4}).$
Likewise, one computes a polynomial $\b i_1(x,y,X,Y)$ 
such that $i_1(k) = \b i_1(x,y,x^k,y^k).$ It can be expressed 
as a sum of $12$ monomials in $X$ and $Y,$ each with a 
coefficient which is a polynomial in $x$ and $y.$  
Thus 
$
i_1(k) 
= \sum_{i=1}^{12}
c_i(x,y) \prod_{j=1}^4 (\mu_{i,j}(x,y))^{k_i},
$
for some polynomials $c_1, \dots, c_{12}$ 
and monomials $\mu_{1,1} \dots, \mu_{12, 4}$ in $x$ and $y.$
Now, 
$$
\sum_{m \in \Z_{\ge 0}^4} ii(m_2, m_1, m_3, m_4)
 t_1^{m_1} t_2^{m_2} t_3^{m_3} t_4^{m_4}$$
$$= \sum_{{A,B,C,D\in \Z}
\atop{A+D,B+C,B+D,C+D\ge 0}
}
i(A,B,C,D) x^{2A+B+C+2D}y^{2A+B+2C+3D}
t_1^{B+C}t_2^{A+D}t_3^{B+D}t_4^{C+D}.
$$
This is a sum of seven subsums corresponding 
to the seven cases which appear in
\eqref{pro:3.1}.  The simplest 
of these is 
$$\begin{aligned}
&\sum_{A,B,C,D \ge 0}i_1(A,B,C,D) 
x^{2A+B+C+2D}y^{2A+B+2C+3D}
t_1^{B+C}t_2^{A+D}t_3^{B+D}t_4^{C+D}
\\&= \sum_{i=1}^{12}
\frac{c_i(x,y)}
{(1-\mu_{i,1}(x,y) t_2x^2y^2)(1-\mu_{i,2}(x,y) t_1 t_3xy)(1-\mu_{i,3}(x,y) t_1t_4xy^2)(1-\mu_{i,4}(x,y) t_2t_3t_4x^2y^3).
}\end{aligned}
$$  
In each of the other six sums one can make a substitution 
to obtain a similar, fourfold sum of $A',B',C',D'$ from $0$ to 
infinity.  
For example, in the second case listed  \eqref{pro:3.1}, one has the conditions
$A,B,C\geq -D\geq 1.$
Substituting 
$
A'= A+D,$ $B' = B+D,$ $C' = C+D,$ and $D' = -D-1,$
yields
$$
\begin{aligned}
&\sum_{A',B',C',D'=0}^\infty i_1(A',B',C',0)
t_1^{B'+C'+2D'+2}t_2^{A'}
t_3^{B'}t_4^{C'}
x^{2A'+B'+C'+4D'+4}y^{2A'+B'+2C'+6D'+6}
\\&=
\sum_{i=1}^{12}
\frac{c_i(x,y)x^4y^6t_1^2}{(1-\mu_{i,1}(x,y)t^2x^2y^2)(1-\mu_{i,2}(x,y)t_1t_3xy)(1-\mu_{i,3}(x,y)t_1t_4xy^2)(1-t_1^2x^4y^6)}
\end{aligned}
$$
The other five subsums are treated similarly.
Totaling up the resulting rational functions and simplifying 
gives
\eqref{power series identity}, 
completing the proof of theorem 
\end{proof}

\subsection{Convergence}
\label{ss: convergence}
In this section, we prove the convergence of 
local zeta integrals 
 \begin{thm}
\label{thm: convergence}
Take $W \in \c W_{\psi_N}(\pi), f \in \Flat(\chi)$ 
and $\phi \in \c S(\Mat_{4\times 2}).$ 
Then  the local zeta integral 
$I(W, f, \phi; s)$ converges 
for $\Re(s_1-s_2)$ and $\Re(s_2)$ 
both sufficiently large.
\end{thm}
\begin{proof}
We need to show that convergence of 
$I_3( W, f; s)$ defined in \eqref{eq: def of I3}, 
for $W \in W_{\psi_N}(\pi)$ and $f$ a smooth 
section of the family of induced representations
$\Ind_P^G(\chi;s).$ To do this, we simply bound 
$f_{\chi;s}$ by a constant times the spherical 
section 
$f_{\Re(s)}^\circ,$
where $\Re(s) \in \R^2$ is the real part of $s.$ 
 Then $\II_2.f_{\chi;s}(\widetilde \delta(t))$ is bounded by a constant multiple of 
$M(w_2^{-1}w_1^{-1}, \Re(s)).f^\circ_{\Re(s)}(w_3\widetilde \delta(t)),$
where $M(w_2^{-1}w_1^{-1},\Re(s))$ is a standard intertwining 
operator. 
The unramified character $(\chi_0; \Re(s))\delta_{B_G}^{-1}$ may be
identified with an element $\varsigma$ of $X(T_G)\otimes_Z\R.$ 
The integral defining the standard intertwining 
operator converges provided the canonical pairing 
$\la \varsigma, \alpha^\vee \ra$ is positive for 
all positive roots $\alpha$ with 
$w_2^{-1} w_1^{-1} \alpha < 0.$ 
Inspecting this set of roots, one finds it is convergent provided
$\Re(s_2)>1, \Re(s_1-s_2)> 2,$ and $\Re(s_1+s_2) > 5.$
Moreover, it converges to a section of the 
representation induced (via normalized induction)
from $\varsigma^{w_1w_2}.$ 

Next, we need to understand the dependence of 
$M(w_1w_2, \Re(s)).f^\circ_{\Re(s)}(w_3\widetilde \delta(t))$ on $T.$ 
In order to do this, we write 
$w_3 \wt \delta(t) $ 
as $\wt \nu(t) \wt \tau(t) \wt \kappa(t)$ 
where $\wt\nu(t) \in  U, \ \wt \tau(t) \in  T_G$
and $\wt \kappa(t)$ varies in a compact set. It is 
convenient to do so using the 
basic algebraic substitution
\begin{equation}\label{eq:algIwasawa}
\bpm 1 & 0 \\ r & 1 \epm 
= \bpm r^{-1} & 1 \\ 0 & r \epm 
\bpm 0 & -1 \\ 1 & r^{-1} \epm,
\end{equation}
which corresponds to the Iwasawa decomposition 
if $F$ is nonarchimedean, but remains valid in the 
Archimedean case as well. 

Recall that $\wt \delta(t)$
is the product of $x_{37}( -b),$ $x_{36}( -c)$
and  $x_{25}(-d) 
 x_{48}(-a)
x_{29}(ad),$
which all commute with one another. 
We can 
 partition $T$ into $16$ subsets 
 and use the identity \ref{eq:algIwasawa} to obtain 
 a uniform expression for $\wt \tau(t)$ on each 
 subset, and compute $\varsigma^{w_1w_2}\delta_{B_G}^{1/2}( \wt \tau(t))$ in each case,
 obtaining 
$$
\left \{ \begin{array}{lr}
1, & |b| \le 1,\\
|b|^{-2u_1}, & |b| > 1
\end{array}\right \}\times 
\left \{\begin{array}{lr} 1, & |c|\le 1,\\
|c|^{-2u_2}, & |c| > 1
\end{array}\right \}
\times
\left \{ \begin{array}{lr}
1, & |a|, |d| \le 1,\\
|d|^{-2u_2}, & |a|\le 1, |d| > 1,\\
|a|^{-2u_1}, & |a| > 1, |ad| \le 1,\\
|a|^{-2u_1-u_2}|d|^{-2u_2}, & |a|> 1, |ad| >1\end{array}
\right \},
$$
where $u_1:= \Re(\frac{s_1 -s_2-2}2),\ u_2:= \Re(\frac{s_1+s_2-4}2).$
Note that most these contributions 
are already visible in  
 \eqref{pro:3.1}. 
Moreover, as in \eqref{eq:11} we have 
$
|\nu_s(t)| = |a|^{2u_1} |b|^{u_1} |c|^{u_2} |d|^{u_1+u_2} 
$

Next, we consider the quantity $W(t)
\delta_B^{-1/2}(t)$ which appears in the integral
\eqref{eq: def of I3}. Using 
\cite{cass-shal} or \cite{lapid-mao} in the nonarchimedean 
case, or \cite{Wallach1}, and \cite{Wallach2} as 
explicated in \cite{Soudry-Memoir} and \cite{Soudry-ASENS95} in the archimedean 
case, we have  
\begin{equation}
\label{W= sum x phi}
W(t)
\delta_B^{-1/2}(t) = \sum_{x \in X_{\pi}} \Phi_x(ad,bc,bd, cd)
x(t),
\end{equation}
 where $X_\pi$ is a finite set of finite functions depending on the representation $\pi,$ and $\Phi_x$ is a Bruhat-Schwartz function $F^4 \to \C$ for 
each $x.$

Thus we obtain a sum of integrals of the form 
\begin{equation}
\label{eq: int D phi(t) x(t) ... dt}
\int_D \Phi(ad, bc,bd,cd) x(t) 
|a|^{k_1u_1+l_1u_2} |b|^{k_2u_1+l_2u_2} |c|^{k_3u_1+l_3u_2} |d|^{k_4u_1+l_4u_2}
 \, dt,
\end{equation}
where $D$ is one of our $16$ subsets,
$k_1, \dots, k_4$ and $l_1, \dots, l_4$ 
are explicit integers depending only on $D,$   
$\Phi$ is a Bruhat-Schwartz function $F^4 \to \R,$
and $x$ is a real-valued finite function $Z\bs T \to \R.$

Now, for each of the seven cases which appear 
in \eqref{pro:3.1},  make a change of 
variables, as in the unramified case so that $|a|^{k_1u_1+l_1u_2} |b|^{k_2u_1+l_2u_2} |c|^{k_3u_1+l_3u_2} |d|^{k_4u_1+l_4u_2}$ is 
expressed as powers of the absolute values of the 
new variables. For example, 
when $|d|>1$ and $|a|,|b|,|c| \le 1,$ we have
$|\nu_s(t)| |d|^{-\Re(s_1)-\Re(s_2)+4}
=|a|^{2u_1} |b|^{u_1} |c|^{u_2} |d|^{u_1-u_2},$
and substitute $a'=ad,  b'=bd, c'=cd$ 
and $d'=d^{-1}.$ After the substitution, 
each exponent is a nontrivial non-negative
linear combination of $u_1$ and $u_2.$ 
Also $|d'|$ is bounded, 
and we have $\Phi(a', b'c'(d')^2, b',c'),$
which provides convergence as $|a'|, |b'|$ 
or $|c'| \to \infty.$  
It follows that the integral converges provided
$u_1$ and $u_2$ are sufficiently large, relative 
to the finite function $x.$ 

As a second example, we consider the 
case $|a|,|c|>1, |b|,|d|\le 1.$ In this case, 
we make the change of variables
$b'=bc, d'=acd, a'=a^{-1}, c'=c^{-1}.$ 
We obtain the integral
$$
\int\limits_{|a'|<1, |c'|< 1, |b'c'|\le 1, |a'c'd'|\le 1}\hskip -.6in
\Phi(c'd', b', a'b'( c')^2d', a'd') x(t)
|a'|^{u_1+u_2}|b'|^{u_1}|c'|^{2u_1+2u_2}|d'|^{u_1+u_2} \, da'\, db'\, dc'\, dd',
$$
assuming that $u_1$ and $u_2$ are positive and
sufficiently large (depending on $x$), the integrals
on $a'$ and $c'$ are convergent due to 
the domain of integration, and the integrals 
on $b'$ and $d'$ are convergent from the decay 
of $\Phi.$ Indeed, 
$\Phi(c'd', b', a'd', a'(b')^2 c'd')
\ll |a'b'c'(d')^2|^{-N}$ for any positive integer $N$ 
because $\Phi$ is Bruhat-Schwartz, and then 
$ |a'b'c'(d')^2|^{-N}\le |b'd'|^{-N}$ on the domain $D.$ 
The other five cases appearing in \eqref{pro:3.1}
are handled similarly. 

The nine cases which do not appear in  \eqref{pro:3.1} are easier. For example suppose that 
$|c|$ and $|d|$ are both $>1$ while $|a|$ and $|b|$ 
are both $\le 1.$ Then the exponents 
of $|a|$ and $|b|$ are the same as in  $\nu_s(t),$
i.e., they are $2u_1$ and $u_1$ respectively. 
This gives convergence of the integrals on $a$
and $b$ when $\Re(u_1)$ is sufficiently large (relative
to $x$). 
The integrals on $|c|$ and $|d|$ converge
because of the rapid decay of $\Phi$ in $cd.$ 
The other eight cases are treated similarly, completing
 the proof of the convergence of $I_1(W, f, \phi; s).$
Now consider $I_1(R(k).W, R(k).f, \omega_\psi(k).\phi; s).$ 
Each bound used in the analysis of $I_1$
 can be made uniform 
as $k$ varies in the compact set $k.$ 
Hence $I_1(R(k).W, R(k).f, \omega_\psi(k).\phi; s)$
varies continuously with $k$ so its integral is again 
absolutely convergent. 
\end{proof}

\subsection{Continuation to a slightly larger domain}
In this section, we 
prove that the local zeta integral $I(W,f, \phi;s)$ 
extends analytically to a domain that 
includes the point $s_1=5, s_2=1.$ This point is 
of particular interest for global reasons. 
We keep the notation from the previous section.
There are two issues. The 
first is related to the convergence 
of the integral $\II_2.f_{\chi;s}.$ As we have 
seen, this integral is {\it not} 
absolutely convergent at $(5,1).$ We must show
that it extends holomorphically to a domain containing 
$(5,1).$ 
Then we need to prove convergence of 
the integral over $Z\bs T.$ 
 The  
domain of absolute convergence for this integral
depends on the exponents of the representation 
$\pi.$ To make this precise, we use terminology 
and notation from \cite{Asg-Sha}, 
section 3.1.

\begin{pro}\label{pro:extension to (5,1)}
Suppose that $\Pi$ satisfies $H(\theta_4)$ 
and that $\tau_1$ and $\tau_2$ satisfy 
$H(\theta_2)$ (as in \cite{Asg-Sha}, 
section 3.1). 
Then for any $\varepsilon > 0,$ the local zeta integral 
 $I(W, f, \phi; s)$ 
extends holomorphically to all $s \in \C^2$ satisfying
$\Re(s_1-s_2) \ge \max( 2 \theta_4+2\theta_2+2,3)+\varepsilon, 
\ \Re(s_1+s_2) \ge 5+\ve,$ 
$\Re(s_2) \ge \frac 12 + \ve,$
$\Re(s_1+2s_2) \ge 2 \theta_2+1.$
\end{pro}
\begin{proof}
We first need to extend 
$\II_2.f_{\chi;s}$ beyond its domain of absolute
convergence. 
It suffices to do this for flat $K$-finite sections, 
even though the convolution sections encountered in section 
\ref{ss: initial computations} are not, in general, flat of 
$K$-finite. Indeed, the integral operator
$\II_2$ commutes with the 
convolution operators considered in \ref{ss: initial computations}. Moreover, these operators are 
rapidly convergent, and hence preserve holomorphy.

As we have seen in the unramified 
computation $\II_2$ can be expressed 
$\II_3 \circ M(w_1^{-1}, \chi;s),$
where $\II_3$ is an operator defined on  
$\nInd_{B_G}^G((\chi;s )\delta_{B_G}^{-1/2})^{w_1}$
by the $u_2$ integral in  \eqref{int U0 = int U1 int U2}.
Then, $M(w_1^{-1},  \chi; s)$ is absolutely convergent for 
$\Re(s_2) > \frac 12, \Re(s_1-s_2) > 1, 
\Re(s_1+s_2-3) > 3.$
If we insert absolute values into the integral which defines
$\II_3,$ we obtain a standard intertwining operator 
attached to $w_2^{-1}.$ We 
may write is as a composite of rank one intertwining 
operators attached to 
$\{ \alpha > 0: w_2^{-1} \alpha < 0\}.$ 
The rank one operator attached to $\alpha$ 
is absolutely convergent provided that 
$ \la \alpha^\vee, \varsigma^{w_1}\rangle$ is positive.
Running through the eight relevant roots, we find 
that only one rank one operator 
diverges at $(5,1).$ It is attached to the simple root
$\alpha_3$ which satisfies  $\la \alpha^\vee, \varsigma^{w_1}\rangle=2\Re(s_2)-2.$

Thus, we only need to improve our treatment 
of the integral over a single one-parameter 
unipotent subgroup.
Thus, we consider 
\begin{equation}
\label{GL2 Jac}
\int_F f^{w_1}_{\chi;s}(w[3]x_{34}(r) g) \psi(c_4r) \, dr,
\end{equation}
where $c_4 \in F$ and 
$f^{w_1}_{\chi;s}$ is a section of the family 
$\nInd_{B_G}^G((\chi; s)\delta_{B_G}^{-1/2})^{w_1}, \ s \in \C^2.$ 
Notice that \eqref{GL2 Jac} may be regarded as a
Jacquet integral for the rank-one Levi attached to the simple 
root $\alpha_3.$ 
By \cite{JacquetsThesis}, this extends to an entire
function of $s$ when $f^{w_1}_{\chi;s}$ is flat. 
If we apply it to the output of $M(w_1^{-1},  \chi; s),$ then it has no poles other
than those of $M(w_1^{-1},  \chi; s).$ Now we use again 
the fact that the asymptotics of a Whittaker function, 
are controlled by the exponents of the relevant 
representation. This time we apply it to the induced representation of our rank one Levi. For 
most values of $s,$ the exponents 
are $((\chi; s)\delta_{B_G}^{-1/2})^{w_1}$
and $((\chi; s)\delta_{B_G}^{-1/2})^{w_1w[3]}$
and the Whittaker function is bounded in 
absolute value by a linear combination 
of spherical vectors. 
 
On the line $s_2=1,$ this may fail: 
if $((\chi; s)\delta_{B_G}^{-1/2})^{w_1}=((\chi; s)\delta_{B_G}^{-1/2})^{w_1w[3]},$ then 
an extra log factor appears in the asymptotics of the Whittaker  function (cf. \cite{Goldfeld-Hundley}, 6.8.11, 
for example).  Bounding $\log |x|$ by $|x|^{-\varepsilon}$
with $\ve > 0$
as $x \to 0,$
in this case, we again bound the integral \eqref{GL2 Jac}
by a sum of spherical sections. In fact, the extra $|x|^\varepsilon$ may be safely ignored, since we obtain convergence for $s$ in an open set and $\ve>0$ can be taken 
arbitrarily small. Thus, if $w_2=w[3]w_2',$ then 
$\II_2.f_{\chi;s}$ extends holomorphically to 
the domain 
where the standard intertwining operator
attached to $w_2'$ converges 
on both $f^\circ_{\Re(s), w_1},$ 
 and $f^\circ_{\Re(s), w_1w[3]}.$
Inspecting $\{ \alpha > 0 : (w_2')^{-1} \alpha < 0 \},$
we see that this means
 $\Re(2s_2-1), \Re(s_1-s_2-3)$
and $\Re(s_1+s_2-5)$ must all be positive.
As a side effect, we find that  
 $|\II_2.f_{\chi;s}(g)|$ is bounded by 
 a suitable linear combination of 
$M(w_2^{-1}w_1^{-1}, \Re(s)).f^\circ_{\Re(s)}$
and $M((w_2')^{-1}w_1^{-1}, \Re(s)).f^\circ_{ \Re(s)}.$

As before, we obtain a sum of integrals of the 
form \eqref{eq: int D phi(t) x(t) ... dt} where 
now the integers $k_1, \dots, k_4$ and
$l_1, \dots , l_4$ depend on the choice of 
domain $D$ and on a choice of 
 between 
$w_2$ and $w_2'.$ 

In order to obtain a precise domain of convergence, 
we need information about the finite function $x.$ 
Firstly, since we have taken absolute values and
assumed unitary central character, it factors through 
the map $t \mapsto (|ad|, |bc|, |bd|, |cd|).$ We may assume 
that $x$ is given in terms of real powers of the coordinates
and non-negative integral powers of their logarithms, since
such functions span the space of real-valued finite functions.
Since a power of $\log y$ may be bounded by an arbitrarily 
small positive (resp. negative) power of $y$ as $y\to \infty$
(resp. $0$), for purposes of determining the domain of 
convergence, we may assume that there are no logarithms.
Thus we may assume 
$x(t) = |ad|^{\rho_1} |bc|^{\rho_2} |bd|^{\rho_3} |cd|^{\rho_4}$
with $\rho_1,\rho_2,\rho_3,\rho_4 \in \R.$ 
The quadruples $(\rho_1,\rho_2,\rho_3,\rho_4)$ which appear 
are governed by the exponents of $\pi,$ by 
\cite{cass-shal} or \cite{lapid-mao} in the nonarchimedean
 case, 
and \cite{Wallach1}, \cite{Wallach2} 
(see also \cite{Soudry-ASENS95})
in the archimedean
case.
Hence they are bounded in absolute value
by $\max(\theta_2, \theta_4),$
by the definition of $H(\theta_2)$
and $H(\theta_4)$ in \cite{Asg-Sha} and the 
bound on exponents of tempered 
representations found in \cite{Wallach2}, 
Theorem 15.2.2 in the archimedean case, 
or \cite{Waldspurger} in the nonarchimedean
case, we see that $|\rho_1|, |\rho_2| \le \theta_4, 
|\rho_3|, |\rho_4| \le \theta_1.$  

What remains is a careful case-by-case analysis 
along the same lines as the proof of convergence. 
For each choice of $D,$ after a suitable change 
of variables we have an integral which is convergent
provided $u_1$ and $u_2$ are sufficiently 
large, and ``sufficiently 
large'' is given explicitly in terms of $\rho_1, \dots, \rho_4.$ 

For example, the above integral corresponding to the 
case $|a|,|c|>1$ and $|b|,|d|\le 1$ will now feature
a Schwartz function integrated against
$$|c'd'|^{\rho_1} |b'|^{\rho_2}  |a'b' (c')^2d'|^{\rho_3}|a'd'|^{\rho_4}
|a'|^{u_1+u_2}|b'|^{u_1}|c'|^{2u_1+2u_2}|d'|^{u_1+u_2},$$
and so will converge provided
$u_1+u_2+\rho_3+\rho_4$, 
$u_1+\rho_2+\rho_3$, $2u_1+\rho_1+2\rho_2,$ and $u_1+u_2+\rho_1+\rho_3+\rho_4$ are all positive
 \end{proof}

\section{Local zeta integrals II}\label{section: local zeta integrals ii}
In this section we continue our study of 
 the local zeta integral $I(W, f, \phi; s)$ 
at the ramified places. 
Write $U_P^-$ for the unipotent radical of the parabolic opposite $P.$
Notice that $PU_P^-w$ is a Zariski open subset of $GSO_{12}.$ We say that 
$f \in \Flat(\chi)$ 
is {\bf simple} if it is
supported on $PK_1$ where $K_1$ is a compact subset of $U_P^-w.$ 

\begin{pro}
\label{pro: nonvanishing}
Suppose that $f$ is simple. 
Then $I(W,f, \phi; s)$ has meromorphic 
continuation to $\C^2$ for each $\phi \in \c S(\Mat_{4\times 2})$ 
and each $W\in \c W_{\psi_N}(\pi)$
Moreover, if  $s_0$ is an element of $\C^2,$ 
then there exist
$W, f$ and $\phi$ such that 
$I(W, f, \phi; s_0) \ne 0.$ 
\end{pro}
\begin{proof}
We begin with some formal manipulations which are
valid over any local field.
The process requires many of the same subgroups which 
were defined during the proof of theorem \ref{thm: unfolding}, 
and we freely use notation from that section.
$$
I(W, f, \phi; s) 
= \int_{Z U_4\bs C_Q} \int_{U_Q^w \bs U_Q} 
\int_{\Mat_{1\times 2}} 
W(g) f(wug, s) [\omega_\psi(ug).\phi]\bpm r \\ I_2 \\ 0 \epm 
\, \ol \psi(r \cdot \bbm 1\\ 0 \ebm )\, dr\, du \, dg.
$$
Define $U_1(a,b,c)$ as in \eqref{eq:6}. Then
$$
 [\omega_\psi(ug).\phi]\bpm r \\ I_2 \\ 0 \epm 
 = [\omega_\psi(U_1(r_1, r_2, c) ug).\phi](\Xi_0),
$$
(for any $c$) where $\Xi_0:=\bspm 0 \\ I_2 \\ 0\espm$ and $r = (r_1\ r_2).$
Also $W(U_1(r_1, r_2, c)g) = \ol \psi(r_1) W(g) 
= \ol \psi(r \cdot \bbm 1\\ 0 \ebm ) W(g).$ 
Hence $$
I(W, f, \phi; s) 
= \int_{Z U_5\bs C_Q} \int_{U_Q^w \bs U_Q} 
W(g) f(wug, s) [\omega_\psi(ug).\phi](\Xi_0) 
du \, dg,
$$
where $U_5$ is the product of $U_6, U_2$ and the center $Z(U_1)$
of $U_1.$ 

Recall that $C_Q^w$ is a standard parabolic subgroup 
of $C_Q.$ Let $U(C_Q^w)^-$ denote the unipotent radical 
of the opposite parabolic.  
Then $C_Q^w \cdot U(C_Q^w)^-$ is a subset of full 
measure in $C_Q$ and we can factor the Haar measure on 
$C_Q$ as the product of (suitably normalized) left Haar measure
on $C_Q^w$ and Haar measure on $U(C_Q^w)^-.$
 Hence 
 $$
I(W, f, \phi; s) 
= \int_{Z U_5\bs C_Q^w}
\int_{U(C_Q^w)^-}
 \int_{U_Q^w \bs U_Q} 
W(gu_1) f(wugu_1, s) [\omega_\psi(ugu_1).\phi](\Xi_0) 
du \,du_1\, d_\ell g$$
Conjugating $g$ across $u,$ making a change of variables, 
and making use of the equivariance of $f$ yields 
$$
I(W,f, \phi; s) 
= \int_{U(C_Q^w)^-}
 \int_{U_Q^w \bs U_Q} 
 J( R(u_1).W, \omega_\psi( uu_1).\phi; s)
 f(wuu_1, s)\, du\, du_1.
$$
where 
$$
J(W, \phi, s) =  \int_{Z U_5\bs C_Q^w}W(g) [\omega_\psi(g).\phi](\Xi_0) 
\ \on{Jac}_1(g) (\chi; s)(wgw^{-1})\, d_\ell g.
$$
with 
$\on{Jac}_1(g)$ being the Jacobian of the change of variables in 
$u.$ 
Now conjugation by $w$ maps $U_Q^w \bs U_Q \times U(C_Q^w)$ 
isomorphically onto the unipotent radical $U_P^-$ of the parabolic 
opposite $P.$ Hence, if $\Phi$ is any smooth compactly 
supported function on $U_Q^w \bs U_Q \times U(C_Q^w)^-,$
then there is a flat section $f$ such that 
$f(wuu_1, s_0 ) = \Phi(u, u_1).$ 

We claim that the integral
$J(W, \phi; s)$ 
 converges provided $\Re(s_1-s_2)$ and $\Re(s_2)$ 
are both sufficiently large, and that
 $J( R(u_1).W, \omega_\psi( uu_1).\phi; s)$ 
extends meromorphically to $\C^2$ and is a continuous function 
of $uu_1$ away from the poles. 
Granted this claim, is clear that if the 
integral of $J( R(u_1).W, \omega_\psi( uu_1).\phi; s_0)$ 
against the arbitrary test function $ f(wuu_1, s_0)$ 
is always zero, then $J( W, \phi; s_0)$
is zero for all $W$ and $\phi.$ 

Now, $C_Q^w=(P_1\times P_2)^\circ$ is the intersection of $C_Q$ with the product
of the Klingen parabolic $P_1$ of $GSp_4$ and the Siegel parabolic $P_2$
of $GSO_4.$ Let $C'=(P_1\times M_2)^\circ$ 
denote the subgroup of elements
whose $GSO_4$ component lies in the Levi, and let $U_5'=C' \cap U_5=U_6Z(U_1).$ Then $C'$ surjects onto $ZU_5 \bs C_Q^w,$ which is thus 
canonically identified with $ZU_5' \bs C'.$ 
Expressing the measure on $C_Q^w$ in terms of Haar measures
on $U_1, U_2$ and $(M_1\times M_2)^\circ,$ 
and then identifying $Z(U_1) \bs U_1$ with $\Mat_{2\times 1}(F)$
via the map $\bar u_1(\bbm r_1 & r_2\ebm) = U_1(r_1, r_2, 0),$ 
yields the following expression for $J(W, \phi; s)$:
\begin{equation}\label{eq: first alt exprn for J(W, phi; s)}
 \int\limits_{\Mat_{1\times 2}} \int\limits_{(M_1\times M_2)^\circ}
 W(\bar u_1(r) m) [\omega_\psi(\bar u_1(r)m).\phi](\Xi_0) 
\ \on{Jac}_1(m) (\chi; s)(wmw^{-1})\, 
\delta_{C_Q^w}^{-1}(m) 
dm\, dr,
\end{equation}
where  
$\delta_{C_Q^w},$ 
is the modular quasicharacter.

Now, elements
of $C'$ map under $j$ into the Siegel Levi of $Sp_{16}.$ So that
$$
[\omega_\psi\circ j(c').\phi](\xi) = |\det c'|^{\frac 12}
\phi( \xi\cdot c'), 
$$
where $\cdot$ is the rational right action of $C'$ on $\Mat_{4\times 2}$
by 
$$
 \xi\cdot m_Q^1( g_1, g_2) = g_1^{-1} \xi g_2.
$$
(with $m_Q^1$ as in \eqref{eq: def mQ1(g1, g2)}).
The stabilizer of the matrix $\Xi_0$
is precisely the group 
$M_5$ introduced in the proof of theorem \ref{thm: unfolding}.

In \eqref{eq: first alt exprn for J(W, phi; s)}, 
conjugate $m$ across $\bar u_1(r),$  make a change of 
variables in $r,$ and let $\Jac_2(m)$ 
denote the Jacobian.
Define 
$$
\mu_s(m) 
= (\chi;s )(wmw^{-1}) 
\delta_{C_Q^w}^{-1}(m) 
|\det m|^{\frac 12}
\Jac_1(m) 
\Jac_2(m).
$$
Then 
replace $m$ by $m_5 m_5'(g)$ where $m_5 \in M_5$ and 
$m_5'(g) =  m(1, I_2, g),$
(with $m$ as in \eqref{m: GL2 x GL2 x GL1 -> M1 x M2 circ}).
Observe that 
$$
\Xi_0 \cdot m_5 m_5'(g) \bar u_1(r) 
= \bpm r \cdot g \\ g \\ 0 \epm.
$$
Hence if $x(g,r) = m_5'(g) \bar u_1(r g^{-1}),$ then 
\begin{equation}
\label{eq: J(W, phi, s)}
J(W, \phi; s)=
 \int\limits_{\Mat_{1\times 2}} 
 \int\limits_{GL_2}
J'(R(x(g,r)).W, s) 
 \phi\bpm r \\ g\\ 0 \epm 
\mu_s( m_5'(g)) |\det g|^{-1}
\, dg\, dr,
\end{equation}
\begin{equation}
\label{eq: J'(W,s)}
\text{ where } \qquad 
J'(W, s):= \int\limits_{ZU_6 \bs M_5}
 W(m_5) 
 \mu_s( m_5) 
\, dm_5.\end{equation}
Direct computation shows that
$\mu_s(m_5'(g)) = |\det g|^{s_2} \chi_2(\det g).$

Write  $M_5=U_6 T_5 K_5,$
where $T_5 = T \cap M_5$ and 
$K_5$ is the maximal compact subgroup of the $GL_2$ factor. 
Then
$$J'(W, s):= \int_{K_5}
\int_{Z \bs T_5}
 W(tk) 
 \mu_s(t)\delta_{B_5}^{-1}(t)  
$$
where $\delta_{B_5}$ is the modular quasicharacter of the 
standard Borel subgroup $B_5$ of $M_5.$ 
Set $t_6'(a) = \diag(a,1,1,a^{-1}, 1,1,1,1,a,1,1,a^{-1}),$
and write $t \in T_5$ as $t_6 t_6'(a)$ for 
$t_6 \in T_6$ and $a \in F^\times.$ 
Then 
$$J'(W, s) = \int_{K_5} 
\int_{F^\times}  J''( R(t_6'(a)k).W, s) \mu_s(t_6'(a)) \, dt,$$
$$\text{ where } 
J''(W, s) = \int_{Z\bs T_6} W(t_6)  \mu_s(t_6)\delta_{B_5}^{-1}(t_6)  
\, dt_6.$$
Observe that $J'(W,s)$ may be written formally as 
$$
\int_{M_6 \bs M_5} J''(R(g_1).W, s) 
\mu_s(g_1) dg_1.
$$
Also, direct computation 
shows that 
$\mu_s(t_6'(a)) = |a|^{s_1-s_2-4} \chi_1(a)/\chi_2(a).$

For $\phi_1$ a smooth function of compact support 
$F^2 \to \C$ let 
$$
[\phi_1*_1W](g) = \int_{F^2} W(g U_1(r_1,r_2,0)) \phi(r_1,r_2) \, dr.
$$
Observe that 
$$
[\phi_1*_1W](M_5(t,g_3)) = W(M_5(t,g_3))  
\widehat \phi_1( g_3^{-1} \cdot \bbm 1\\ 0 \ebm t \det g_3).
$$
Thus, by replacing $W$ by $\phi_1*_1 W$
(which is justified by \cite{Dixmier-Malliavin}) we may introduce 
what amounts to an arbitrary test function on $M_6\bs M_5.$ 
Hence 
$$
J'(W, s_0 ) = 0 \forall W \iff J''(W, s_0) = 0 \forall s_0.
$$
Similarly, if 
$$
[\phi_2*_2 W](g) := \int_{U_6} W(g u_6 ) \phi_2(u_6)\, du_6,\qquad 
\phi_2 \in C_c^\infty( U_6), 
$$ 
then 
$J''(\phi_2*_2W, s_0) = 0 \forall \phi_2\in C_c^\infty(U_6)$
if and only if $W$ vanishes identically on $T_6.$
In particular, if $J''(W,s_0)$ vanishes identically 
on $\whitt,$ then $\whitt$ is trivial-- a contradiction. 

This completes the formal arguments for proposition \ref{pro: nonvanishing}.
It remains to check the convergence and continuity statements 
made above. These will be proved based on facts about 
asymptotics of Whittaker functions and the Mellin transform 
\begin{equation}
\label{eq: mellin transform}
Z_{\xi, n}. \Phi(u) = \int_{F^\times} \Phi(x) \xi(x) (\log|x|)^n |x|^u \, d^\times x,
\end{equation}
where 
$\Phi \in \c S(F), \ \xi: F^\times \to \C, \text{ a character,}\ 
n \ge 0 \in \Z, \  
u \in \C.$
We recall some properties. 
\begin{pro}\label{pro: properties of the mellin transform}
Fix a character $\xi$ and a non-negative integer $n.$
\begin{enumerate}
\item There is  a real number $c$ depending 
on $\xi$ such that 
the integral defining $Z_{\xi, n}. \Phi$ converges 
absolutely and uniformly on 
 $\{ u \in \C : \Re(u) \ge c+\ve\}$ for all $\ve > 0$
 and all $\Phi \in \c S(F).$ 
\item There is a discrete subset $S_\xi$ of $\C$ such that 
$Z_{\xi, n}.\Phi$ extends meromorphically to all of $\C$ with
poles only at points in $S_\xi.$ Moreover, there is an 
integer $o_{\xi, n}$ such that no pole of $Z_{\xi, n}.\Phi$ 
has order exceeding $o_{\xi, n},$
for any $\Phi.$ 
\item If $F$ is archimedean, then $Z_{\xi, n}.\Phi = Q_\Phi( q^{u})$ for 
some rational function $Q.$
\item We have 
\begin{equation}
\label{Z n+1 = D.Zn}
Z_{\xi, n+1}.\Phi(u) = \frac{d}{du} Z_{\xi, n}.\Phi(u).
\end{equation}
\end{enumerate}
\end{pro}
\begin{proof} 
If $n= 0$ then the first three parts are proved in \cite{tatesthesis}.
Convergence for $n>0$ is straightforward, 
since $\log |x|$ grows slower than any positive power of $|x|$ at infinity, 
and slower than any negative power as $|x| \to 0.$ 
Equation \eqref{Z n+1 = D.Zn} is clear in the domain of convergence, 
and follows elsewhere by continuation.
The first three parts for general $n$ then follow.
\end{proof}
Next, we need a version of the 
expansion \eqref{W= sum x phi}.
Specifically, if we replace $W(t)$ by $W(tk)$ 
then each $\Phi_x$ in \eqref{W= sum x phi}
will be in $\c S(F^4 \times K).$
(See  \cite{Soudry-Memoir}, especially the remark on p.20.)

Let us now consider the convergence issues raised by our 
formal computations more carefully.
Recall that  $I(W, f, \phi; s)$ was initially 
 expressed as an integral 
over $U_Q^w \bs U_Q \times ZU_5 \bs C_Q.$ 
In the course of our arguments, we have expressed it 
as an iterated integral over
$$
(U_Q^w \bs U_Q \times U(C_Q^w)^-)\times 
(F^2 \times GL_2)
\times (F^\times \times K_5)
\times Z\bs T_6.
$$
In order to perform the integration on $Z\bs T_6$ we may identify with 
with $\{ \bar t_6(a) = \diag(1,1,a^{-1},a^{-1},1,a^{-1},1,a^{-1},1,1,a^{-1},a^{-1}), a \in F^\times\}.$
Then
$\mu_s( \bar t_6(a))= |a|^{\frac{s_1-s_2}2-2} \frac{\chi_1^2\chi_2}{\chi_3}(a).$

Now, the integral $J''(W,s)$ is the Mellin transform taken 
along the one dimensional torus we use to parametrize $Z\bs T_6.$
Its convergence and analytic continuation follow 
directly from proposition \ref{pro: properties of the mellin transform}
and 
\eqref{W= sum x phi}.
Now consider  
$J''(R(t_6'(a)k).W, s)$  with $a \in F^\times, k \in K_5.$
We claim that it is  smooth 
and of rapid decay in $a.$
In the domain of absolute convergence, this 
is easily seen.
For $s$ outside of the domain of convergence, 
we use \eqref{Z n+1 = D.Zn} to 
pass to the Mellin transform 
of a suitable derivative at a point {\it inside} 
the domain of convergence. 
To get $J'(W,s)$ we integrate $k$ over the compact set $k$
and take  another
Mellin transform in the variable $a.$ 
This of course converges absolutely,
and by similar reasoning, we see that
$J'(R(x(g,r).W), s)$ is smooth.
Now, set $g$ equal to $\bspm t_1& b \\ 0 & t_2 \espm k$ and 
consider the integral 
$$
\int\limits_{\Mat_{1\times 2}}
\int\limits_{F^\times}
\int\limits_{F^\times}
\int\limits_{F}
\int\limits_{K_{GL_2}}
J'( R(x(g , r)).W, s)
\phi\bpm r\\ g\\ 0 \epm 
\chi_2( t_1 t_2)
|t_1|^{s_4}
|t_2|^{s_5}
\, dk \, db\, dr \, d^\times t_1 \, d^\times t_2,
$$
where $s_4$ and $s_5$ are two more complex
variables, and $K_{GL_2}$ is the standard maximal compact subgroup 
of $GL_2.$ The integrals on 
$k, r$ and $b$ converge absolutely and 
uniformly because $K_{GL_2}$ is compact 
and $\phi$ is Schwartz-Bruhat. 
The integrals on 
$t_1$ and $t_2$  take two more 
Mellin transforms, yielding a meromorphic
function of four complex variables. The restriction 
to a suitable two-dimensional subspace of $\C^4$
 is $J(W, \phi, s).$ 
Moreover, $J(R(u_1).W, \omega_\psi(uu_1).\phi, s)$ 
remains continuous in $u_1 \in U(C_Q^w)^-$ and
$u \in U_Q^w \bs U_Q,$ which completes the proof. 
\end{proof}

\section{Global Identity}\label{section: global identity}
We now return to the global situation.  Thus $F$ is again a number field with adele ring $\A,$ while 
and $\psi_N,$ and $\Flat(\chi),$ 
are defined as in section \ref{sec: global}
and \ref{sec: notation}, respectively.
 In addition, let 
$\pi = \otimes_v \pi_v$ be an irreducible,
globally $\psi_N$-generic
cuspidal automorphic representation 
of $GSp_4(\A) \times GSO_4(\A),$ 
with normalized central character $\omega_\pi,$
and $\varphi$ be a 
cusp form from the space of $\pi,$ 
etc.

For $r$ a representation of $^LG$ define 
$L(u, \pi, r\times \eta)$ to be the twisted $L$ 
function.  Thus at an unramified place 
$v$ 
the local factor is 
$$
L_v( u, \pi_v, r \times \eta_v)
= \det( I - q^{-u} \eta_v(\f w_v) r( \tau_{\pi_v}))^{-1},
$$
where $\f w_v$ is a uniformizer, 
$q_v$ is the cardinality of the residue class field, 
$\tau_{\pi_v}$ is the semisimple conjugacy class attached to $\pi_v,$ and $\eta_v$  is the local 
component of $\eta$ at $v.$

\begin{thm}\label{thm:global id}
Suppose that $f_\chi=\prod_{v}f_{\chi_v}\in \Flat(\chi), \phi=\prod_v \phi_v\in \c S(\Mat_{4\times 2}(\A))$ 
and $W_{\varphi} = \prod_v W_v$ (the Whittaker function of $\pi$ as in theorem \ref{thm: unfolding}
are factorizable. 
Let 
$I(f_{\chi_v;s}, W_v, \phi_v)$ 
be the local zeta integral, defined as in 
\eqref{eq: local zeta integral def},
and let $S$ be a finite set of places $v$
and all data is unramified for all $v\notin S.$
Then for $\Re(s_1-s_2)$ and $\Re(s_2)$ 
both sufficiently large,
the global integral 
$I(f_{\chi;s}, \varphi, \phi),$
defined as in \eqref{eq: global integral definition}, is equal to 
 $$\frac{L^S(\frac{s_1-s_2}2-1, \pi, St_{GSp_4}^\vee
\otimes St_{GL_2^{(1)}} \times \frac{\chi_3}{\chi_1 \chi_2^2}) 
L^S(\frac{s_1+s_2}2-2, \pi, St_{GSp_4}^\vee
\otimes St_{GL_2^{(2)}} \times \frac{\chi_3}{\chi_1 \chi_2})}{N^S(s,\chi)}
$$
times 
$$
\prod_{v\in S} I(f_{\chi_v}, W_v, \phi_v), 
$$
where $N^S(s, \chi)$ 
is the product of partial zeta functions 
corresponding to \eqref{eq:N(s, chi)}
\end{thm}
\begin{rem}
Let $\eta_1$ and $\eta_2$ be any two characters
of $F^\times \bs \A^\times.$  Fix $\pi$ 
and let $\omega_\pi$ be its central 
character.  Then the system
$$
\frac{\chi_1^3 \chi_2^3}{\chi_3^2} = \omega_\pi, 
\qquad 
\frac{\chi_3}{\chi_1 \chi_2^2} = \eta_1, 
\qquad
\frac{\chi_3}{\chi_1 \chi_2} = \eta_2
$$
has a unique solution.
If $\eta_1 = \eta_2$ is trivial, then 
it is given by $\chi_1 =\chi_3 = \omega_\pi$ and 
$\chi_2 \equiv 1.$
\end{rem}
\begin{proof}
It follows from theorem \ref{thm:  main local}
the bound obtained in 
\cite{langlands-eulerproducts}
that for any cuspidal representation $\pi= \otimes_v \pi_v$ of $GSp_4(\A) \times GSO_4(\A)$ the infinite product
$\prod_{v\in S}  I( f_{\chi_v}, W_v, \phi_v)$ is convergent 
for $\Re(s_1-s_2)$ and $\Re(s_2)$ sufficiently large.
It then follows from theorem \ref{thm: unfolding}, and the basic results on integration 
over restricted direct products presented in 
\cite{tatesthesis}
that 
$$
I(f_{\chi;s}, \varphi, \phi) 
= \prod_v I( f_{\chi_v;s}, W_v, \phi_v),
$$
which, in conjunction with theorem \ref{thm:  main local} again gives the result.
\end{proof}
\begin{cor}Let $\pi_v$ be the local 
constituent at $v$ of a 
cuspidal representation $\pi.$ Then 
the local zeta integral 
$I_v(W_v, f_v, \phi_v; s)$ has meromorphic continuation to $\C^2$ 
for any $W_v, f_v$ and $\phi_v.$
\end{cor}
\begin{proof}
This follows from a globalization argument. 
We create a section of the global induced representation which is $f$ at one 
place and simple at every other place. 
Meromorphic continuation of the global zeta integral follows from 
that of the Eisenstein series. 
Having shown meromorphic continuation at every other place in proposition \ref{pro: nonvanishing}, we 
deduce it at the last place.
\end{proof}

\section{Application}\label{section: application}
In this section we give an application relating 
periods, poles of $L$ functions, and functorial 
lifting. The connection between 
$L$ functions and functorial lifting 
in this case was obtained in \cite{Asg-Sha}.

Let $\Pi$ be a globally generic cuspidal automorphic 
representation of $GSp_4,$ and let $\tau_1$ and $\tau_2$ be two 
cuspidal automorphic representations of $GL_2.$  Assume that $\Pi,$ 
$\tau_1$ and $\tau_2$ have the same 
central character. 
Then  $\tau_1 \otimes\tau_2$ may be regarded as a 
representation of $GSO_4$
via the realization of $GSO_4$ as a 
quotient of $GL_2 \times GL_2$ discussed in section \ref{section: unramified calculation}, 
and when $\Pi \otimes \tau_1 \otimes\tau_2$
is restricted to the group $C_Q$
(which we identify $C_Q$ with subgroup of 
$GSp_4 \times GSO_4$ as in section \ref{sec: notation} its central character is trivial.

  Now 
take $s_1$ and $s_2$ to be two complex numbers.  
Let $\chi_1 = \chi_2=\chi_3$ be trivial.  Consider the space
$\on{Flat}(\chi)$ of flat sections as in section \ref{sec: notation}.  Its elements are 
functions $\C^3 \times G(\A) \to \C,$ but we
regard each as a function $\C^2 \times G(\A) \to \C$ 
by pulling it back through the function 
$(s_1, s_2) \mapsto (s_1, s_2, \frac{3s_1+3s_2}{2}).$
Then theorem 
\ref{thm:global id} relates the global integral 
\eqref{eq: global integral definition}
with the product of $L$ functions
$$L^S( \frac{s_1-s_2-2}{2}, \wt \Pi\times \tau_1)
L^S( \frac{s_1+s_2-4}{2}, \wt\Pi\times \tau_2).$$

For $f \in \on{Flat}(\chi),$
let $$\b r(f,g) = \lim_{s\to 
(5,1)
} \left( \frac{s_1+s_2-4}{2} -1 \right) \left( \frac{s_1-s_2-2}{2}-1\right)E(f_{\chi;s},g)
$$
be the residue of the Eisenstein series at $(5,1).$ 
As $f$ varies we obtain a residual automorphic representation 
which we denote $\c R.$ Given 
$\b r \in \c R$ and $\phi \in \c S(\Mat_{4\times 2}(\A)),$
we define the Fourier coefficient 
$\b r^{\theta(\phi)}$ 
exactly as in \ref{eq: FC of E}.
Varying $\b r$ and $\phi$ we obtain a space 
of smooth, $K$-finite  functions
of moderate growth $Z(\A) C_Q(F)\bs C_Q(\A) \to \C.$
We denote this space $FC(\c R).$ 
Write $V_{\Pi}$ for the space of the representation $\Pi$
and $V_\tau$ for that of $\tau.$ 
Then, define the period
$\c P: V_{\Pi} \times V_\tau \times FC(\c R) \to \C,$
by the formula
$$
\c P(\varphi_{\Pi}, \varphi_\tau, \b r^{\theta(\phi)} )
= \int_{Z\bs C_Q} 
\b r^{\theta(\phi)}(g) \varphi_\Pi(g_1) \varphi_\tau(g_2) \, dg.
$$

\begin{thm}
First suppose that $\tau_1 \ne \tau_2.$ 
 Then the following are equivalent: 
\begin{enumerate}
\item  $L^S(s, \wt \Pi\times \tau_1)$ and 
$L^S(s, \wt\Pi \times \tau_2)$ have poles at $s=1.$ 
\item $\Pi$ is the weak lift of $\tau_1\times \tau_2$ 
\item the period $\c P$ does not vanish 
identically on $V_\Pi \times V_\tau\times FC(\c R).$
\end{enumerate}
Similarly, if $\tau_1=\tau_2,$ then the following are equivalent:
\begin{enumerate}
\item $L^S(s, \Pi\times \tau_1)$ has a pole at $s=1.$ 
\item $\wt\Pi$ is the weak lift of $\tau_1\times \tau'$
for some cuspidal representation $\tau'$ of $GL_2(\A),$  
\item the period $\c P$ does not vanish 
identically on $V_\Pi \times V_\tau\times FC(\c R).$
\end{enumerate}
\end{thm}
\begin{proof}
The relationship 
between poles and the similitude theta correspondence
was established 
in \cite{Asg-Sha}.
What is new here is the period condition, which follows from our
earlier results. Indeed, 
for $f \in \Flat(\chi), \phi \in \c S(\Mat_{4\times 2}(\A))$
$\varphi_\Pi \in V_\Pi$ and $\varphi_\tau \in V_\tau,$ 
the period 
$\c P(\varphi_\Pi, \varphi_\tau, R(f)^{\theta(\phi)})$
vanishes if and only if $$\lim_{s\to 
(5,1)
} \left( \frac{s_1+s_2-4}{2} -1 \right) \left( \frac{s_1-s_2-2}{2}-1\right)I(f_{\chi; s} , \varphi_\Pi, \varphi_\tau, \phi) \ne 0.$$
By \cite{Asg-Sha}, the local components of $\Pi$
all satisfy $H(15/34)$ and the local components of $\tau$
all satisfy $H(1/9).$ 
Hence each ramified local zeta integral 
is holomorphic at $(5,1)$ 
by proposition \ref{pro:extension to (5,1)}.
Moreover, by proposition \ref{pro: nonvanishing},
each ramified local zeta integral 
is nonzero at $(5,1)$ for a suitable choice of data.
The result follows.
\end{proof}
\section{A similar integral on $GSO_{18}$}
\label{section: GSO18}
In this section we consider the global integral \eqref{eq: global integral definition} in the case $n=3.$
  Our 
unfolding does not produce an integral
involving the Whittaker functions attached to 
our cusp forms, but it does reveal another 
intriguing connection with the theta correspondence. 

As before, the space of double cosets $P \bs GSO_{18} / R_Q$
is represented by elements of the Weyl group, and 
$$
I(f_{\chi;s}, \vph,  \phi) = \sum_{w\in P \bs GSO_{18} / R_Q}
I_w(f_{\chi;s}, \vph,  \phi),\qquad\text{ where }
$$
$$
I_w(f_{\chi;s}, \vph,  \phi)
= \il{C_Q^{w}(F) \bs C_Q(\A)} 
\vph(g) \il{U_Q^{w}(\A) \bs U_Q(\A) }
f_{\chi;s}(wu_2 g) \iq{U_Q^{w}}
\theta(\phi, u_1u_2g) \, du_1 \, du_2\, dg,
$$
which is zero if $\psi_l \big|_{Z(U_Q) \cap U_Q^w}$ is nontrivial, or 
if some parabolic subgroup of $C_Q$ stabilizes the flag
$0 \subset \ol U_Q^w \subset (\ol U_Q^w)^\perp$ in 
$U_Q/Z(U_Q),$ where $\ol U_Q^w$ is the image of $U_Q^w$ and 
$(\ol U_Q^w)^\perp$ is its perp space relative to the symplectic
form defined by composing 
$l:Z(U_Q) \to \G_a$ with the commutator map $U_Q/Z(U_Q)\to Z(U_Q).$

\begin{lem}
Let $w_\ell$ denote the longest 
element of the Weyl group of $GSO_{18},$ 
let $w_1$ be the shortest element of
$(W \cap P) \cdot w_\ell\cdot (W \cap Q)$ and let 
$w_2 = w_1 \cdot w[32].$  Then 
$P w_2 R_Q$ is a Zariski open subset
of $GSO_{18}.$ 
\end{lem}

\begin{pro}
$I_w(f_{\chi;s}, \vph,  \phi)$ is zero unless $w$ is in the open 
double coset.
\end{pro}

\begin{pro}
Let $U_0 \subset C_Q$ be given by 
$$
\left\{ 
u_0(x,x'):=
\left( 
\bpm 1  & x_1 & x_2 & x_3&x_4&x_5 \\ & 1 &0&0&x_6 &*\\ &&1&0&0 &*\\
&&&1&0&*\\ &&&&1&*\\ &&&&&1 \epm, \; 
\bpm 1& x_1'&x_4' & x_7 & x_8 & *\\ &1 & x_6' & x_9 &*&* 
\\&&1&0&-x_9&*\\&&&1&-x_6&*\\&&&&1&-x_1\\&&&&&1
\epm
\right)
\;:\;\bm
x \in \G_a^9, \\ \\ x' \in \G_a^3\em
\right\},
$$
where entries marked $*$ are determined by symmetry, 
and 
for $x \in \A^9$ and $x' \in \A^3,$ 
let 
$$\psi_{U_0}(u_0(x,x')) = 
\psi( x_1+x_6-x_1'-x_6' +x_9).
$$
Let $SL_2^{\alpha_3}$ be the copy of $SL_2$
generated by $U_{\pm \alpha_{3}},$ and let 
$R_0$ be the product of $U_0$ and $SL_2^{\alpha_3}.$  
Let $\psi_{R_0}$ be the character of $R_0$ which restricts
to $\psi_{U_0}$ and to the trivial character of $SL_2^{\alpha_3}.$ 
Let 
$$\vph^{(R_0, \psi_{R_0})}(c)
= \iq{R_0} \vph(rc) \psi_{R_0}(r) \, dr
= \iq{U_0} \iq{SL_2^{\alpha_3}} \vph(uhc) \psi_{U_0}(u) \, dh\, du,\qquad (c \in C_Q(\A)).$$

Let $V_4= \{ u(x,x'): x_i = x_i', i=1,4,6\}\subset U_0,$
and let
$$
\xi_0 = \bpm 1&0&0\\ 
0&1&0 \\ 0&0&0\\0&0&0\\0&0&1\\0&0&0\epm 
\in \Mat_{6\times 3}(F).
$$
Then
$$
I_{w_2}(f_{\chi;s}, \vph,  \phi)
= \il{Z(\A)V_4(\A) \bs C_Q(\A)}
\vph^{(R_0, \psi_{R_0})}(c)
\il{U^{w_2}(\A) \bs U(\A)} 
f_{\chi;s}(w_2 uc) \left[\omega_\psi( uc).\phi\right](\xi_0)\, du
\, dc.
$$
\end{pro}
\begin{rem}
It was shown in \cite{GRS-PeriodsPoles} that the 
period we  obtain in the $GSp_6$ characterizes the 
image of the theta lift from $SO_6$ to $Sp_6.$ 
\end{rem}
\begin{proof}
First,
$
U_Q^{w_2}$ is the set of  all $u_Q( 0,Y,0)$
such that rows $2,$ $5$ and $6$ of $Y$ are zero. 
It follows that 
$$
\theta^{U_Q^{w_2}}( \phi , u_1 u_2 g):=
\iq{U_Q^{w_2}} \theta( \phi , u_1 u_2 g) \, du_1
= \sum_{\xi } [\omega_\psi(u_2 g).\phi]( \xi),
$$
where the sum is over $\xi \in \Mat_{6 \times 3}(F)$ 
such that rows $3,4$ and $1$ are zero.
Next, the identification of $C_Q$ with a subgroup of $GSp_6 \times GSO_6$ identifies $C_Q^{w_2}$ with the subset of elements 
of the form 
\begin{equation}\label{c w coords}
\left( 
\bpm t  & x_1 & x_2 & x_3&x_4&x_5 \\ & a &0&0&b &*\\ &&a'&b' &0 &*\\
&&c'&d'&0&*\\ &c&&&d&*\\ &&&&&t\lambda \epm, \; 
\bpm g&W \\ & _tg^{-1}\lambda \epm
\right), 
\; 
t \in GL_1, \; \bpm a&b\\ c&d \epm,  \bpm a'&b'\\ c'&d' \epm \in GL_2, \;
g \in GL_3.\end{equation}
  
Now we can expand $\vph$ along the abelian unipotent 
 subgroup which consists of elements of the form 
 $u_1(W):=\left( I_6, \;\bspm I_3 & W \\ & I_3 \espm\right),
 \; W \in \mnsw 3.$
 The constant term is of course zero.
 The group $C_Q^{w_2}$ acts transitively on the nontrivial characters.
 As a representative for the open orbit we select the 
 character $\psi_{2,1}(u_1(W)) := \psi(W_{2,1}).$
The stabilizer of this representative 
can be described in terms of the coordinates from \eqref{c w coords}
as the set of 
 elements of $C_Q^{w_2}$ such that $g \in GL_3$  is of the form $\bspm t_1 & u \\ & g_1 \espm$
 with $g_1 \in GL_2$ and $\det g_1 = \lambda.$ 
 Denote this group by $C_1^{w_2}.$
Now we write the integral as a double integral, with the 
inner integral being 
\begin{equation}\label{inner integral}
\iq{\mnsw{3}}\il{U_Q^{w_2}(\A) \bs U_Q(\A) }
f_{\chi;s}( w_2 u_2 u_1(W) g)
\theta^{U_Q^{w_2}}( \phi, u_2 u_1(W) g)\, du_2 \psi^{-1}_{2,1}(W)\, dW.
\end{equation}
Now, 
$$
u_Q( \xi , 0, 0 ) 
u_1(W) = 
u_1(W) u_Q(\xi, 0,0) u_Q( 0,\xi W, -\xi W \, _t\xi ),
\qquad ( W \in \mnsw{3}, \; \xi \in \Mat_{6\times 3}).
$$
It follows that \eqref{inner integral} is equal to 
$$
\iq{\mnsw{3}}\il{U_Q^{w_2}(\A) \bs U_Q(\A) }
f_{\chi;s}( w_2 u_2 g)
\sum_{\xi }
[\omega_\psi(u_2  g).\phi](\xi)
\psi_l( \xi W \,_t\xi) \psi^{-1}_{2,1}(W)\, du_2 \, dW,
$$
with $\xi$ summed over $6\times 3$ matrices such that
rows $3,4$ and $6$ are zero.  
Clearly the integral on $W$ picks off the terms such that
$\psi_l( \xi W \,_t\xi) \psi^{-1}_{2,1}(W)$ is trivial. 
Now, direct calculation shows that 
$$
\xi = \bpm \xi_1 & \xi_2 & \xi_3 \\\xi_4 & \xi_5 & \xi_6 \\ 
0&0&0\\ 0&0&0 \\ \xi_7& \xi_8&\xi_9 \\ 0&0&0\epm, 
\qquad W = \bpm y_1 &y_2 & 0 \\ y_3 & 0& -y_2 \\ 
0 & -y_3&-y_1 \epm,\implies
\psi_l ( \xi W \, _t \xi ) 
= \psi\left( \det \bpm
y_3 & -y_1 & y_2 \\ \xi_4 & \xi_5 & \xi_6 \\  \xi_7& \xi_8&\xi_9
\epm \right).
$$
So, in the coordinates above, the condition for $\psi_l( \xi W \,_t\xi) \psi^{-1}_{2,1}(W)$ to be trivial is $\xi _4 = \xi_7 =0$
and $\det \bpm \xi_5 & \xi_6 \\  \xi_8&\xi_9\epm = 1.$
Observe that if $\xi_1$ is also zero, then 
the function 
$g \mapsto [\omega_\psi(g).\phi](\xi)$ is invariant on the left by 
$\{ ( I_6, \bspm u&\\ & _tu^{-1} \espm )\in C_Q^{w_2}: u = \bspm 1& x&y\\ &1&\\&&1\espm\in GL_3 \}.$
Thus, the contribution from such $\xi$ is trivial by cuspidality.
The group $C_1^{w_2}$ permutes the remaining terms transitively,
and the stabilizer of 
$\xi_0$
is 
$$
C_2^{w_2}:=
\left \{ \left( 
\bpm t  & x_1 & x_2 & x_3&x_4&x_5 \\ & a &0&0&b &*\\ &&a'&b' &0 &*\\
&&c'&d'&0&*\\ &c&&&d&*\\ &&&&&t\lambda \epm, \; 
\bpm g&W \\ & _tg^{-1}\lambda \epm
\right)\in C_1^{w_2}\;: \qquad 
g = \bpm t& x_1&x_4 \\ & a&b \\ & c & d \epm
\right\}.
$$
Expanding first on $x_1$ and $x_4,$ and then on the unipotent
radical of the diagonally embedded $GL_2,$ and using lemma \ref{lem:1} two more times gives the result.
\end{proof}
